\documentclass[12pt,a4paper]{siamart220329}
\usepackage[utf8]{inputenc}
\usepackage[english]{babel}
\usepackage{graphicx}
\graphicspath{{./figures}}

\usepackage{amsfonts, amsmath, amssymb, bm}
\usepackage{hyperref}
\usepackage{todonotes}
\usepackage{adjustbox}
\usepackage{bbm}
\usepackage{aligned-overset}
\usepackage{mathtools}
\usepackage{ntheorem}

\usepackage{algpseudocode}

\usepackage{subcaption}

\crefname{algocf}{algorithm}{algorithms}
\Crefname{algocf}{Algorithm}{Algorithms}

\crefname{algorithm}{algorithm}{algorithms}
\Crefname{algorithm}{Algorithm}{Algorithms}

\crefname{lemma}{lemma}{lemmas}
\Crefname{lemma}{Lemma}{Lemmas}
\crefname{theorem}{theorem}{theorems}
\Crefname{theorem}{Theorem}{Theorems}

\newtheorem{example}{Example}

\crefname{conjecture}{conjecture}{conjectures}
\Crefname{conjecture}{Conjecture}{Conjectures}

\crefname{figure}{figure}{figures}
\Crefname{figure}{Figure}{Figures}

\crefname{example}{example}{examples}
\Crefname{example}{Example}{Examples}

\Crefname{ALC@unique}{Line}{Lines}

\DeclareMathOperator{\vct}{span}
\DeclareMathOperator{\VEC}{vec}
\DeclareMathOperator*{\argmin}{argmin}

\def\lineref#1{\hyperref[#1]{line \ref{#1} of \Cref{#1}}}

\newsiamremark{remark}{Remark}
\crefname{remark}{remark}{remarks}
\Crefname{remark}{Remark}{Remarks}
\title{
  ResQPASS: an algorithm for bounded variable linear least squares with asymptotic Krylov convergence}

\author{Bas Symoens \thanks{Department of Mathematics, Universiteit Antwerpen, Belgium  (\email{bas.symoens@uantwerpen.be})}
 \and Wim Vanroose \thanks{Department of Mathematics, Universiteit Antwerpen, Belgium  (\email{wim.vanroose@uantwerpen.be})}}

\ifpdf
\hypersetup{
	pdftitle={
  ResQPASS: an algorithm for bounded variable linear least squares with asymptotic Krylov convergence},
	pdfauthor={B. Symoens, W. Vanroose}
}
\fi

\begin{document}
\maketitle

\begin{abstract}
	We present the Residual Quadratic Programming Active-Set
        Subspace (ResQPASS) method that solves large-scale linear
        least-squares problems with bound constraints on the
        variables.  The problem is solved by creating a series of
        small problems of increasing size by projecting onto  the basis
        of residuals.  Each projected problem is solved by the
        active-set method for convex quadratic programming,
        warm-started with a working set and solution from the previous
        problem.  The method coincides with conjugate gradients (CG)
        or, equivalently, LSQR when none of the constraints is active.
        When only a few constraints are active the method converges,
        after a few initial iterations, like CG and LSQR. An analysis
        links the convergence to an asymptotic Krylov subspace.  We
        also present an efficient implementation where QR
        factorizations of the projected problems are updated over the inner
        iterations and Cholesky or Gram-Schmidt over the outer iterations.
\end{abstract}



\section{Introduction}
Inverse problems reconstruct an unknown object from measurements based
on a mathematical model that describes the data acquisition
process. They invert the measured signal into information about the
unknown object. Adding prior knowledge to the problem, such as bounds
on the variables, can significantly improve the quality of the
solution.

Many inverse problems, e.g., computed tomography (CT), are formulated
as algebraic reconstruction problems such as linear least squares,
e.g., ${\min_{x\in\mathbb{R}^n}\|Ax-b\|^2_2}$, which minimizes the
mismatch between the data and model. Here, ${A\in\mathbb{R}^{m \times n}}$ models the propagation of X-rays
through the object \cite{kak2001principles}, ${x \in \mathbb{R}^n}$
describes the unknown pixel values, and $b \in \mathbb{R}^m$ is a
vector containing noisy measurements. As $A$ is often ill-conditioned, a
straightforward least-squares solution will be contaminated by the noise
in the measurements \cite{hansen2006deblurring}.  Adding bounds to the
unknown material parameters regularizes the problem.

Similar bounds appear in nonnegative matrix factorization (NMF), where
a matrix is factorized as a product of two nonnegative matrices
\cite{gillis2020nonnegative}.  At each iteration a bounded variable
least squares system needs to be solved, alternating between the two
factors.

Contact problems in mechanics are often modeled as the minimization of
a quadratic energy function, subject to inequalities that represent the
impenetrability of modeled objects \cite{dostal2016scalable}.

The matrices in these problems are often large and sparse. The linear
operators $Av$ and $A^T w$ can both be implemented as sparse
matrix-vector products and their calculation can be performed in
parallel.  This allows the use of matrix-free methods such as Krylov
subspace methods, which rely solely on matrix-vector operations to
find a solution.  However, current Krylov subspace methods are not designed to
deal with bounds on the variables from first principles.   

We propose an efficient solution method for large-scale
\textit{bounded-variable least squares} (BVLS) problem using a
matrix-free subspace method. We solve
\begin{equation}\label{eqn:LS}
	\min\, \frac{1}{2}\|Ax-b\|_2^2 \quad \text{s.t.} \quad {\ell \le x \le u}, 
\end{equation}
where the bounds are vectors $\ell, u \in \mathbb{R}^n$. This BVLS
problem can, for example, be solved by the Stark-Parker algorithm
\cite{stark1995bounded}.  When only nonnegative bounds are present, it
becomes a \textit{nonnegative least squares} (NNLS) problem
\cite{lawson1995solving}. Note that a least squares problem with an
$\ell_1$ regularization term can also be formulated as a nonnegative
quadratic programming (QP) problem \cite{chen2001atomic}.

The result is a constrained quadratic programming (QP) problem with
objective $f(x)\coloneqq \frac{1}{2}\|Ax-b\|_2^2$.

The Karush-Kuhn-Tucker (KKT) optimality conditions
\cite{nocedal1999numerical} for \Cref{eqn:LS} are
\begin{subequations}\label{eqn:kkt}
	\begin{align}
		A^T(Ax-b) - \lambda + \mu & =0 \label{eqn:kkt:stationarity}                                               \\
		\lambda_i(x_i - \ell_i)   & = 0                             &  & i\in\{1,\ldots,n\}\label{eqn:kkt:lambda} \\
		\mu_i(u_i - x_i)          & = 0                             &  & i\in\{1,\ldots,n\} \label{eqn:kkt:mu}    \\
		\ell_i \leq x_i           & \leq u_i                        &  & i\in\{1,\ldots,n\}                       \\
		(\lambda, \mu)            & \geq 0
	\end{align}
\end{subequations}
where $\lambda \in \mathbb{R}^n$ and $\mu\in \mathbb{R}^n$ are vectors
of Lagrange multipliers associated with, the lower and upper bounds.

The state-of-the-art method for solving general large-scale quadratic
programming (QP) or linear programming (LP) problems is the
\textit{interior-point} (IP) method
\cite{nocedal1999numerical,gondzio2012interior}. In each iteration of
this method, a weighted normal equation must be factorized using
Cholesky factorization. A drawback of this approach is that the
factorization cannot be reused in subsequent iterations because of
significant changes in the weights.

In contrast, with an active-set method, the active constraints change
one by one. As a result, the matrices are modified only through
rank-one updates, allowing reuse of the factorization. This often requires many more (but much cheaper) iterations.

By combining active-set and subspace methods, we can take advantage of
iterative subspace methods to handle large-scale problems while also
benefiting from the ability to reuse factorizations. This enables us
to warm-start each iteration from the previous solution.
However, the proposed method is
best suited for applications where the number of variables in the
active set remains small compared to the total number of variables in
the system.

There is a close resemblance to \textit{column generation}
\cite{lubbecke2005selected}, a technique for large-scale LP problems
that exploits the block structure of the matrix. It expresses the
solution as a convex combination of solutions to subproblems. This
results in an iterative method that alternates between solving a
master problem, which finds the coefficients of the convex
combination, and solving a subproblem that expands the vectors for the
convex combination.  Our projected QP problem resembles the master
problem in column generation.

Here, we use a subspace based on the residuals (see \cref{eqn:rk}) of the first equation of KKT
conditions, \Cref{eqn:kkt}, to approximate the solution.  In each iteration, we solve
a small projected problem, similar to the master problem in column
generation.  This projected problem is again a convex QP problem but
now with small, dense matrices, which is easily solved with a refined
version of classical active-set methods for convex quadratic
programming \cite{nocedal1999numerical}. These refinements are
described in \Cref{sec:qr}, and we use the acronym QPAS for this
method.  Since each outer iteration expands the basis with one vector,
we can easily warm-start from the solution of the previous subspace.
The use of residuals as a basis for the expansion is common in Krylov
methods \cite{liesenstrakos}.  Hence, we call the proposed method the
residual quadratic programming active-set subspace method, resulting
in the acronym ResQPASS.

For inverse problems that involve a rectangular matrix $A$,
Golub-Kahan bidiagonalization \cite{golub1965calculating} is used to
generate solutions in the Krylov subspaces,
\begin{equation}
	\mathcal{K}_k(A^TA, A^Tb) = \text{span} \left\{ A^Tb, (A^TA)A^Tb, \ldots, (A^TA)^{k-1}A^Tb \right\}.
\end{equation}

Krylov subspace methods have been extended to include bound constraints before, through the use of Flexible Conjugate Gradient Least Squares (FCGLS), dubbed `box-FCGLS' \cite{gazzola2021flexible}. This method uses a variable preconditioner along with a smart restarting strategy to solve \cref{eqn:LS}. The main difference with our approach is that while ResQPASS borrows many ideas from Krylov subspace methods, the basis that is constructed is not a Krylov subspace as it directly incorporates information about the bounds. Only by the end of the algorithm will the basis expansions approach Krylov-like basis expansions. This step away from true Krylov subspaces allows for more flexibility in constructing a subspace that is optimal for the problem under consideration. 
The proposed method is competitive with existing methods for large-scale problems with only a small subset of active constraints.
There is also literature on enforcing nonnegativity in \cite{nagy2000enforcing, gazzola2017fast}.  

Here, we make the following contributions. We propose a subspace
method that uses the residuals from the KKT conditions \cref{eqn:kkt} 
as a basis.  In each iteration, we solve a small projected convex QP
problem with dense matrices, using only the application of $A$
and $A^T$ on vectors.  When only a few bounds are active at the
solution we observe superlinear convergence.  We explain this behavior
by making the link to Krylov convergence theory. From a certain point
on, the residuals can be written as polynomials of the normal matrix
and the convergence is determined by the Ritz values of the normal
matrix projected onto the sequence of residuals.

In addition, we contribute an efficient implementation that uses
warm-starting as the residual basis is expanded and updates the
factorization of the matrices each iteration.  Additional recurrence
relations provide further improvements in efficiency.

The analysis of the propagation of rounding errors and the final
attainable accuracy is not part of this study and will be the subject
of a future paper.

In \Cref{sec:basis} we derive
the algorithm ResQPASS and prove some properties of the
proposed subspace method.  \Cref{sec:implementation} discusses several ways to speed up
the algorithm, and \Cref{sec:numerical} discusses some synthetic
numerical experiments representative of different application areas. We conclude in \Cref{sec:conclusion}.

\section{Algorithm ResQPASS}\label{sec:basis}
In this section we introduce the methodology and analyse the
convergence.

We propose to solve \Cref{eqn:LS} by iteratively solving a projected
version of the problem.

\begin{definition} \label{def:resqpas}
	The \textit{residual quadratic programming active-set
          subspace} iteration for $A \in \mathbb{R}^{m \times n}$, $b
        \in \mathbb{R}^m$ and $\ell,u \in \mathbb{R}^n$ such that $\ell \leq 0 \leq u$ with associated Lagrange multipliers $\lambda_k,
        \mu_k \in \mathbb{R}^n$, generates a series of approximations
        $\{x_k\}_{k \in \mathbb{N}}$ that solve
	\begin{equation}\label{eqn:xk}
		\begin{aligned}
			x_k = \argmin_{x \in \vct \{r_0,\ldots, r_{k-1}\}} & \|Ax-b\|_2^2  \quad   \text{s.t.}                                                            & \, \ell \leq x \leq u,
		\end{aligned}
	\end{equation}
	where
	\begin{equation}\label{eqn:rk}
		r_k := A^T (Ax_k-b) - \lambda_k +\mu_k.
	\end{equation}
        The feasible initial guess is $x_0=0$, with
        $\lambda_0=\mu_0=0$ and $r_0 :=-A^Tb$.
\end{definition}

\begin{remark}\label{rem:shift}
	The condition $\ell\leq0\leq u$, which ensures feasibility of
        $x_0$, does not imply any restrictions on the problems that
        can be solved. A problem with arbitrary $\ell<u$ can always be
        shifted so that $\ell\leq 0 \leq u$ holds.  This restriction
        is necessary because the active-set method for convex QP
        requires a feasible initial guess (see also
        \Cref{alg:resBasis}).
\end{remark}

A high-level implementation of the algorithm is given in
\Cref{alg:resBasis} and a more detailed implementation in
\Cref{alg:outer}.

Let $V_k \in \mathbb{R}^{n \times k}$ be a basis for
$\text{span}\{r_0,\ldots, r_{k-1}\}$ with $r_k$ defined in
\Cref{eqn:rk} and let $y_k$ be the projection of the approximation
$x_k$ from \Cref{eqn:xk} on this basis (i.e., $x_k=V_ky_k$). The
optimization problem at iteration $k$ can be rewritten as
\begin{equation}\label{eqn:LSproj}
	\begin{aligned}
		\min_{y_k\in\mathbb{R}^k} & \quad\frac{1}{2}\|AV_ky_k-b\|_2^2 \quad \text{s.t.}              & \quad \ell \leq V_ky_k \leq u.
	\end{aligned}
\end{equation}
The KKT conditions for \cref{eqn:LSproj} are
\begin{subequations} \label{eqn:kktP}
	\begin{align}
		V_k^TA^T(AV_ky_k-b) - V_k^T \lambda_k + V_k^T \mu_k & =0 \label{eqn:kktP:stationarity}                                                \\
		(\lambda_k)_i([V_ky_k]_i - \ell_i)                  & = 0                              &  & i\in\{1,\ldots,n\}\label{eqn:kktP:lambda} \\
		(\mu_k)_i(u_i - [V_ky_k]_i)                         & = 0                              &  & i\in\{1,\ldots,n\} \label{eqn:kktP:mu}    \\
		\ell_i \leq [V_ky_k]_i                              & \leq u_i                         &  & i\in\{1,\ldots,n\} \label{eqn:kktP:bound} \\
		(\lambda_k, \mu_k)                                  & \geq 0.
	\end{align}
\end{subequations}
We call the \textit{active set} the set of indices $i\in\mathcal{A}_k
\subset \{1,\ldots,n\}$ where one of the bound constraints
\cref{eqn:kktP:bound} becomes an equality.

As in Krylov-type methods (see \cite{liesenstrakos}), we use residuals
$r_k$ to expand our basis.  For each iteration $k$, the basis of the
subspace is $V_k=\begin{bmatrix}r_0/\|r_0\| & \cdots &
r_{k-1}/\|r_{k-1}\| \end{bmatrix}$.  This choice of residual basis is
natural for Krylov methods such as conjugate gradients (CG). We prove
their pairwise orthogonality in \Cref{thm:ortho}, how this is a
generalization of Krylov in \Cref{thm:generalize}, and how there is
asymptotic Krylov convergence in \Cref{thm:ritzvalue}.

\begin{lemma}[Orthogonality of $r_k$ for optimal solution of projected system]\label{thm:ortho}
	If $y^*_k \in \mathbb{R}^k$ is an optimal solution of
        \Cref{eqn:LSproj} and $x_k=V_ky_k^*$ then the vectors $\{r_k\}_{k \in \{0,1,2, \ldots\}} $
        defined in \Cref{eqn:rk} are pairwise orthogonal.
\end{lemma}
\begin{proof}
	If $y^*_k$ is optimal, the KKT conditions \cref{eqn:kktP}
        hold.  The stationarity conditions
        \Cref{eqn:kktP:stationarity} can be rewritten as 
	\begin{equation}
		V_k^T\left(A^T(AV_ky_k-b) - \lambda_k + \mu_k\right)= V_k^Tr_k = 0.
	\end{equation}
	Since $V_k$ is a basis for $\text{span}[r_0, r_1, \ldots
          r_{k-1}]$, this proves the orthogonality of $r_k$ to $V_k$ for any $k>0$ and hence the orthonormality of $V_k$.
\end{proof}

\begin{lemma}[Orthogonality of $r_k$ for feasible solution of projected system]\label{thm:ortho_feasibility}
  If $y_k, \lambda_k, \mu_k$ is a feasible solution of  
  \ref{eqn:kktP:stationarity} then the $\{r_k\}_{k \in \{1,2 \ldots
    \}}$ defined in \Cref{eqn:rk} are pairwise orthogonal.
\end{lemma}
\begin{proof}
	Again, \Cref{eqn:kktP:stationarity} can be rewritten as an
        orthogonality condition, so a feasible solution of the
        stationarity condition is orthogonal to the previous
        residuals.
\end{proof}

There exists a natural link between the residuals in \Cref{eqn:rk} and
the Krylov subspaces.
\begin{lemma}[Generalization of Krylov subspace]\label{thm:generalize}
	In the unconstrained case, where at every iteration, none of the bounds of
        \Cref{eqn:LSproj} is active, we get
	\begin{equation}\label{eqn:equivKrylov}
		\mathcal{V}_{k+1}\coloneqq\vct\left\{r_0,r_1,\ldots,r_k\right\} \cong \mathcal{K}_{k+1}(A^TA,r_0)\coloneqq \vct\left\{r_0,A^TAr_0,\ldots, \left(A^TA\right)^kr_0\right\},
	\end{equation}
	with $r_k$ defined as in \Cref{eqn:rk}.
\end{lemma}
\begin{proof}
	If the problem is unconstrained, there are no Lagrange
        multipliers $\lambda$ and $\mu$.  If none of the constraints
        is active, the complementarity conditions
        \Cref{eqn:kktP:lambda,eqn:kktP:mu} ensure that
        $\lambda=0=\mu$.  Thus our residuals $r_k$ simplify to
	\begin{equation}
		r_k = A^T(AV_ky_k - b).
	\end{equation}
	For $k=1$, the statement is trivial.
	Assume \Cref{eqn:equivKrylov} holds for $k-1$. Then
	\begin{align}
		r_k & = A^TA\begin{bmatrix} r_0 & r_1 & \cdots & r_{k-1} \end{bmatrix}\begin{bmatrix}y_0\\ \vdots\\ y_{k-1}\end{bmatrix} - A^Tb \\
		    & = A^TA\left( y_0r_0 + \cdots + y_{k-1}r_{k-1}\right) + r_0                                                                \\
		    & = r_0 + y_0A^TAr_0 + \cdots + y_{k-1}A^TAr_{k-1}.
	\end{align}
\end{proof}

\begin{example}\label{ex:cg}
	If we now compare our method applied to the unconstrained
        problem $\min_x\|Ax-b\|_2^2$ with the implementation of LSQR
        \cite{paige1982lsqr} in MATLAB (\verb+lsqr+), we notice in
        \Cref{fig:cg} that they converge similarly. This also suggests
        the possibility of preconditioning, see also \Cref{example:contact}. Note that LSQR is
        mathematically equivalent to CG applied to the normal
        equations $A^TAx = A^Tb$, but LSQR is numerically more
        reliable.

        In this example, $A\in\mathbb{R}^{1000\times600}$ is a sparse
        normally distributed random matrix with density 0.04,
        $x^*\in\mathbb{R}^{600}$ with normally distributed entries,
        and $b = Ax^*$.
        %
\end{example}
\begin{figure}
	\centering
	\includegraphics[scale=0.9]{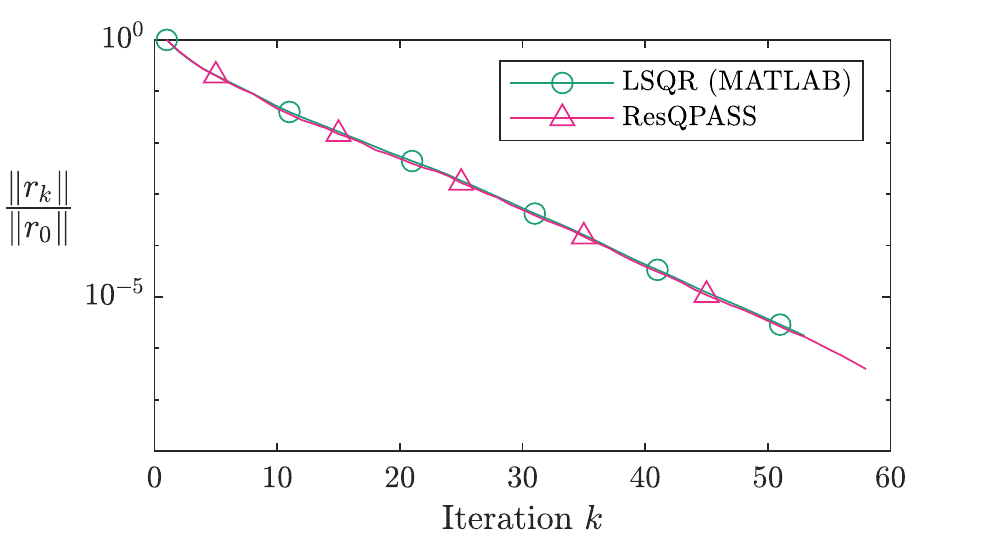}
	\caption{Comparison of convergence of $\frac{\|r_k\|}{\|r_0\|}$ for LSQR and ResQPASS for an
          unconstrained least squares problem.}
	\label{fig:cg}
\end{figure}

\begin{example}\label{ex:tuneable}
	Let us look at an example with a limited number of active
        constraints, where we have control over the maximal number of
        active constraints $i_{\max}$.  Consider the least-squares problem
        $\min\|Ax-b\|_2^2$ with $A\in\mathbb{R}^{1000\times 600}$ a sparse matrix with a density of 4\%. The nonzero elements are all equal to $1$ and are uniformly distributed over the rows and columns of $A$.  
	Let $x^*$ be the unconstrained solution and $b=Ax^*$. Half of the entries of $x^*$ are 0 (uniformly spread over $x^*$) and the other half are $\pm 1$, with equal probability.
	Let $i_{\max}\leq n$ be the maximal number of active constraints. We then add the following constraints to the problem:
	\begin{equation}
		-\frac{1}{2}\|x^*_i\|-0.01 \leq x_i \leq \frac{1}{2}\|x_i^*\|+0.01 \qquad i\in\{1,\ldots, i_{\max}\}.
	\end{equation}
	The $0.01$ offset ensures that the lower and upper bounds are
        not equal. The experiment is performed for
        $i_{\max}=\{0,1,2,4,8,16,32,64,128\}$. From the results in
        \Cref{fig:maxConstr} we conjecture that the method has two
        steps: discovery of the active set and Krylov
        convergence. Note that the discovery phase takes a number of
        iterations roughly equal to the number of active constraints,
        which in this problem is roughly equal to $i_{\max}$ and to
        the number of active constraints ($i_{\max}\approx$ number of
        active constraints in this example). The number of iterations
        for the Krylov convergence is always about the same as for the
        unconstrained problem (where there is no discovery phase).
	\begin{figure}
		\includegraphics[width=\textwidth]{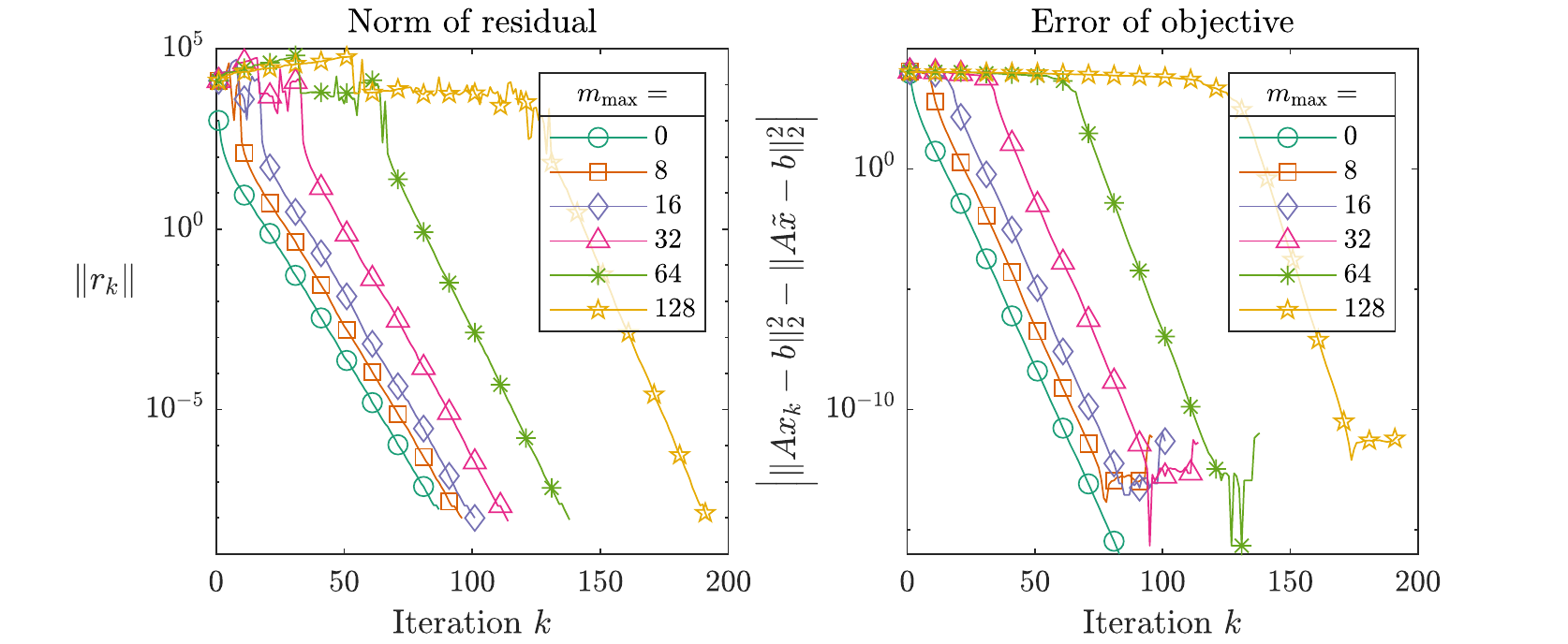}
		\caption{This figure illustrates the convergence
                  behavior for different numbers of active constraints.
                  The residual and objective behave similar to the
                  unbounded ($i_{\max}=0$, Krylov convergence) case,
                  with a delay that is roughly equal to $i_{\max}$,
                  the number of active constraints in the problem. $\tilde{x}$ is an `exact' solution found by applying MATLAB's \texttt{quadprog} with a tolerance of $10^{-15}$.}
		\label{fig:maxConstr}
	\end{figure}

\end{example}

\subsection{Convergence theory}\label{sec:convergence}

\begin{algorithm}
	\caption{Residual quadratic programming active-set subspace  (ResQPASS)}
	\label{alg:resBasis}
	\begin{algorithmic}[1]
		\Require{$A \in \mathbb{R}^{m \times n},b \in \mathbb{R}^m$, $\ell,u \in \mathbb{R}^n, tol>0$}
		\State $r_0 =A^Tb$
		\State $V_1 =r_0/\|r_0\|$
		\State $y_1 = 0$
		\State $\mathcal{W}_1 = \emptyset$
		\For{$k=1,2,\ldots,m$}
		\State $y_k^*, \mathcal{W}_k^*,\lambda_k, \mu_k \gets$ Solve \Cref{eqn:LSproj} using \textproc{qpas}, with initial guess $y_k$ and initial working set $\mathcal{W}_{k}$
		\State $r_k = A^T\left(AV_ky_k^* - b\right)-\lambda_k+\mu_k$
		\If{ $\|r_k\|_2 \le tol$}
		\State $x = V_k y_k,$  break;
		\EndIf
		\State $V_{k+1} \gets \begin{bmatrix} V_k & r_k/\|r_k\| \end{bmatrix}$
		\State $y_{k+1} \gets \begin{bmatrix} (y_k^*)^T & 0 \end{bmatrix}^T$
		\State $\mathcal{W}_{k+1} \gets \mathcal{W}_k^*$
		\EndFor
	\end{algorithmic}
\end{algorithm}
Based in these observations we develop a convergence theory.

\begin{lemma}\label{thm:conjecture}
	Let $\lambda_k, \mu_k \in \mathbb{R}^n$ be the solution for
        the Lagrange multipliers of the projected KKT conditions \cref{eqn:kktP} for iteration $k$. Let $V_k \in  \mathbb{R}^{n \times k}$ be the subspace generated by the
        residuals $r_k$ as in \Cref{def:resqpas}. Then there exists an iteration $0< k_0 \le n$ such that $-\lambda_k+\mu_k \in \vct([V_{k_0}, A^TAV_{k_0-1}])$ for all iterations $k\geq k_0$.
\end{lemma}

\begin{proof}
	$V_k$ can maximally grow to $n$ independent vectors and then ${\text{span}\{V_n\} = \mathbb{R}^n}$.
	In that case, the solutions $-\lambda + \mu  \in \mathbb{R}^n$ of \cref{eqn:kkt} are in $\text{span}(V_n)$.

	The number of active constraints is determined by the number of nonzeros
	in $-\lambda + \mu$ and it is usually much smaller then $n$.

	Because we solve \cref{eqn:kktP} with a basis, the number of
        active bounds in $\ell \le V_ky_k \le u$ depends on $k$. 
		Typically, the number of nonzeros in $-\lambda_k+\mu_k$ grows, initially, proportional with $k$.
	We can construct the following subspace of $\mathbb{R}^n$:
	\begin{equation}\label{eqn:spanlagrange}
		\text{span}\left\{-\lambda_1 + \mu_1,-\lambda_2 + \mu_2, \ldots, -\lambda_k + \mu_k \right\},
	\end{equation}
	where the number of nonzeros in each vector increases initially with $k$.  
	  At some point, this space spans the exact active set of the problem, and also $-\lambda^*+\mu^*$, the
        solution of the full problem \cref{eqn:kkt}, can be written
        as a linear combination of the vectors in this subspace.  Once
        these vectors span the space of the active set, this basis
        does not expand anymore and linear dependence appears.

	By construction, see definition of the residual vector
        \cref{eqn:rk}, the vectors in \cref{eqn:spanlagrange} are
        equivalent to the vectors in
	\begin{equation}
		\begin{aligned}
			\text{span}\left\{r_1-A^T(AV_1y_1-b),r_2-A^T(AV_2y_2-b), \ldots, r_k-A^T(AV_ky_k-b)  \right\}.
		\end{aligned}
	\end{equation}
	As soon as linear dependence appears in
        \cref{eqn:spanlagrange}, the last vector $-\lambda_k+\mu_k$
        can be written as a linear combination of the previous
        vectors. This means that there are coefficients $a_j \in
        \mathbb{R}$ such that
	\begin{equation}
		-\lambda_k + \mu_k = r_k-A^TAV_ky_k = \sum_{j=1}^{k-1} (r_j- A^T(AV_jy_j-b))  a_j  \in \text{span}[V_{k}, A^TAV_{k-1}],
	\end{equation}
	where we use that $-A^Tb = V_{k_0}\|A^Tb\| e_1$, since we
        start with $x_0 =0$.

	At the first iteration $k$ where this linear dependence
        appears, set $k_0:=k-1$.  
\end{proof}

\Cref{thm:conjecture} means that for an iteration $k \ge k_0$ there is an $\alpha^{(k)}\in \mathbb{R}^{k_0}$ and a $\beta^{(k)} \in
\mathbb{R}^{k_0-1} $ such that we can write the residual as
\begin{equation} \label{eq:assumption}
	\begin{aligned}
		r_{k} & =  A^T(AV_k y_k -b) - \lambda_k +\mu_k                                                 \\
		      & = A^T A V_{k} y_k + V_{k_0}\alpha^{(k)} + A^TA V_{k_0-1}\beta^{(k)}                    \\
		      & =A^T A V_{k}(y_k + (\beta^{(k)}, 0)^T) + V_{k_0} \alpha^{(k)} \quad \forall k \ge k_0.
	\end{aligned}
\end{equation}
Here we use again that $A^Tb = V_{k_0}\|A^T b\|e_1$, since we start
with $x_0 =0$. As $\lambda_k$ and $\mu_k$ can still change with the
iteration, the coefficients $\alpha^{(k)}$ and $\beta^{(k)}$ depend on
$k$.

\begin{lemma}\label{lemma:polynomials}
	If $y_k \in \mathbb{R}^k$ and $\lambda_k, \mu_k \in
        \mathbb{R}^n$ are the solution of \cref{eqn:kktP} and the
        iteration $k \ge k_0$ is such that \Cref{thm:conjecture} holds,
        then there are $\gamma_i^{(k)} \in \mathbb{R}^{k_0}$ with $i
        \in \{0, \ldots,k-k_0 \}$ such that the residual is
	\begin{equation}
		r_k  = A^T(AV_k y_k -b) - \lambda_k +\mu_k = \sum_{j=0}^{k-k_0+1} (A^TA)^{j} V_{k_0} \gamma^{(k)}_{j} \quad \forall k \ge k_0.
	\end{equation}
\end{lemma}

\begin{proof}
	Since \Cref{thm:conjecture} holds, the residual
	${r_k =  A^T(AV_k y_k -b) - \lambda_k +\mu_k}$ can be rewritten for iteration $k_0$ as
	\begin{equation}
		\begin{aligned}
			r_{k_0} & = A^TAV_{k_0}y_{k_0} + V_{k_0}\alpha^{(k_0)} + A^TA V_{k_0-1}\beta^{(k_0)}                                                           \\
			        & = A^TA V_{k_0} \left(y_k + (\beta^{(k_0)} 0)^T \right) + V_{k_0-1}\alpha^{(k_0)} = \sum_{j=0}^{1} (A^T A)^j V_{k_0}\gamma^{(k_0)}_j,
		\end{aligned}
	\end{equation}
	where we can use $\gamma^{k_0}_1 = y_{k_0} + (\alpha_2^{(k_0)}
        0)^T $ and $\gamma^{(k_0)}_0 = \alpha_1^{(k_0)}$. So the lemma
        holds for $k=k_0$.
	The residual is added to the basis, which becomes $V_{k_0+1} = [V_k \,\,
          r_{k_0}/\|r_{k_0}\|]$.
	The solution for iteration $k_0+1$ is {$V_{k_0+1}
          y_{k_0+1} = V_{k_0}(y_{k_0+1})_{1:k_0} + r_{k_0}/\|r_{k_0}\|
          (y_{k_0+1})_{k_0+1}$} and can be rewritten as a linear
        combination of the vectors $A^TA V_{k_0}$ and $V_{k_0}$.  The residual for iteration $k_0+1$ is then
	\begin{equation}
		r_{(k_0+1)} =
		(A^TA)^2 V_{k_0} \gamma^{(k_0+1)}_2 + (A^TA) V_{k_0}\gamma^{(k_0+1)}_1 + V_{k_0}\gamma^{(k_0+1)}_0.
	\end{equation}
	Hence for each next iteration $k > k_0$ we have
	\begin{equation}\label{eq:power}
		\begin{aligned}
			r_k & = (A^TA)^{k-k_0+1} V_{k_0} \gamma^{(k)}_{k-k_0+1} +  (A^TA)^{k-k_0} V_{k_0} \gamma^{(k)}_{k-k_0}
			+ \dots + V_{k_0} \gamma^{(k)}_0                                                                      \\
			    & =  \sum_{j=0}^{k-k_0+1}  (A^TA)^{j} V_{k_0} \gamma^{(k)}_{j}.
		\end{aligned}
	\end{equation}
\end{proof}

We now define the subspace $W_{k_0,l}$ as the space spanned by $l>0$
vectors after the linear dependence appeared in the Lagrange
multipliers for iterations $k_0 < k$, including the vector
$v_{k_0}$. It is
\begin{equation}
	W_{k_0,l}:=\text{span}\left\{v_{k_0},v_{k_0+l-1}, \ldots,   v_{k}\right\} =  \text{span}\left\{r_{k_0-1}, \ldots, r_{k_0 + l -2}\right\},
\end{equation}
where $k = k_0 + l-1$.

We introduce the operator $B_{k_0,l}$ on this subspace
$W_{k_0,l}$. The action of $B_{k_0,l}$ on $W_{k_0,l}$ restricted to
$W_{k_0,l}$ corresponds to the action of $A^TA$ on $W_{k_0,l}$.  The
operator $B_{k_0,l}$ is fully determined by its action on the basis
vectors of $W_{k_0,l}$. But instead of the basis vectors we use the
vectors $(A^TA)^m V_{k_0}\eta $ with $m \in \{1,\ldots, l\}$ and an
arbitrary choice of $\eta \in \mathbb{R}^{k_0}$.

We then have equalities that should hold for any choice of $\eta \in \mathbb{R}^{k_0}$:
\begin{equation}\label{eqn:operator_B}
	\begin{aligned}
		B_{k_0,l} V_{k_0} \eta   & = (A^TA)  V_{k_0}\eta,                         \\
		B^2_{k_0,l} V_{k_0}\eta  & = (A^TA)^2 V_{k_0}\eta,                        \\
		                         & \vdots                                         \\
		B^l_{k_0,l} V_{k_0} \eta & =  W_{k_0,l} W_{k_0,l}^T (A^TA)^l V_{k_0}\eta, \\
	\end{aligned}
\end{equation}
where in the last equality we project again on the subspace
$W_{k_0,l}$.  Since these equalities must hold for any $\eta$, we can
replace them with matrix equalities.  Indeed, since we can choose
$\eta$ to be the unit vectors, we obtain the matrix equalities
\begin{equation}
  \begin{aligned}
    B^m_{k_0,l} V_{k_0}  & = (A^TA)^m V_{k_0},  \quad \forall m \in \{0,\ldots,l-1\},\\
    B^l_{k_0,l} V_{k_0} & =  W_{k_0,l} W_{k_0,l}^T (A^TA)^l V_{k_0}.
  \end{aligned}
\end{equation}
This is similar to the Vorobyev moment problem \cite{liesenstrakos}.

\begin{lemma}
	Let $B_{k_0,l} = W_{k_0,l} \,T_{l,l}\,W_{k_0,l}^T$ with
        $T_{l,l} \in \mathbb{B}^{k_0 \times k_0}$ be the operator defined
        by \cref{eqn:operator_B}.  Let $\{r_0,r_1,
        \ldots, r_k \}$ be the series of residual vectors from system
        \cref{eqn:kktP}.  Then there is a polynomial $P_l(t)$ of
        order $l$ with $P_l(0)=1$ such that the projection of $r_k$ on
        $W_{k_0,l}$ satisfies
	\begin{equation}
		W_{k_0,l}^T r_k = P_l(T_{l,l})e_1=0.
	\end{equation}
\end{lemma}

\begin{proof}
	Because $r_k \perp V_k$, we also have $r_k \perp W_{k_0,l}$, since $W_{k_0,l} \subset V_k$. We can then write for $k \ge k_0$,
	\begin{equation}
		\begin{aligned}
			0 = W_{k_0,l}^T r_{k} & = W^T_l \sum_{j=0}^{l} (A^TA)^j V_{k_0} \gamma^{(k)}_j  =  \sum_{j=0}^{l} W^T_{k_0,l} (A^TA)^j V_{k_0} \gamma^{(k)}_j                      \\
			                      & =   \sum_{j=0}^{l} W^T_{k_0,l} B_l^j V_{k_0} \gamma^{(k)}_j =    \sum_{j=0}^{l} W_{k_0,l}^TW_l T_{l,l}^j W_{k_0,l}^T V_{k_0} \gamma^{(k)}_j \\
			                      & =     \sum_{j=0}^{l} T_{l,l}^j (\gamma^{(k)}_j)_{k_0} e_1 = P_l(T_{l,l})e_1,
		\end{aligned}
	\end{equation}
	where we use \cref{eqn:operator_B}
\end{proof}

\begin{lemma}
	Let $k > k_0$,  assume \Cref{thm:conjecture} holds, and let $W_{k_0,l} = \vct\{v_{k_0},\ldots, v_k\}$. Then
	\begin{equation}
		W_{k_0,l}^T A^T A W_{k_0,l} = T_{l,l}
	\end{equation} is tridiagonal.
\end{lemma}

\begin{proof}
	The subspace $W_{k_0,l}$ is spanned by the residuals
        $\text{span}\{ r_{k_0-1}, \ldots, r_{k_0+l-2}\}$, which is
        equivalent to $\text{span}\{v_{k_0},\ldots, v_{k_0 +l-1}\}$.
        Based on \Cref{lemma:polynomials}, each of these can be
        written as
	\begin{equation}
		v_{k_0+t+1}\|r_{k_0+t}\|=  r_{k_0+t} = \sum^{t+1}_{j=0} (A^T A)^j V_{k_0} \gamma^{(k_0+t)}_j,
	\end{equation}
	for $t \in \{0,\ldots, l-2\}$.

	When we apply $A^T A$ to these residuals, we have
	\begin{equation}
		A^T A r_{k_0+t} = \sum^{t+1}_{j=0} (A^T A)^{j+1} V_{k_0} \gamma^{(k_0+t)}_j  \in \text{span} \left\{V_{k_0}, r_{k_0-1}, \ldots,r_{k_0+t+1}  \right\}.
	\end{equation}
	Since $\text{span}\left\{r_{k_0-1}, \ldots,r_{k_0+t} \right\}
        = \text{span}\left\{v_{k_0},\ldots,v_{k_0+t+1}\right\}$ and
         $W_{k_0,l}$ is orthogonal (because it is a subset of
        $V_k$), we can write
	\begin{equation}
		A^TA W_{k_0,l}   = W_{k_0,l+1} H_{l+1,l},
	\end{equation}
	where $H_{l+1,l} \in \mathbb{R}^{(l+1)\times l}$ has a
        Hessenberg structure.  This structure naturally appears when
        the residuals are used as basis.  Now, if we project back on
        $W_{k_0,l}$ we get
	\begin{equation}
		W_{k_0,l}^T A^T A W_{k_0,l} = H_{l,l},
	\end{equation}
	so that $H_{l,l}$ is symmetric and therefore
        tridiagonal. We denote it by $T_{l,l}$.
\end{proof}

\begin{lemma}
	\label{thm:ritzvalue}
	The polynomial $P_l(t)$ has zeros at the Ritz values $\theta_j^{(l)}$ of $T_{l,l}$.
\end{lemma}

\begin{proof}
	The proof follows section 5.3 of \cite{van2003iterative} or
        theorem 3.4.3 of \cite{liesenstrakos}.  The eigenvectors
        $y_j^{(l)}$ of tridiagonal $T_{l,l}$ span
        $\mathbb{R}^l$.  We can then write $e_1 = \sum_j \gamma_j
        y_j^{(l)}$.

	When we assume that $\gamma_j = (y_j^{(l)},e_1) =0$, we obtain
	\begin{equation}
		T_{l,l}y_j^{(l)} = \theta^{(l)}_j y_j^{(l)}.
	\end{equation}
	Since $(T_{l,l}y^{(l)}_j,e_1) = (y^{(l)}_j,T_{l,l}e_1) = 0$, this implies that $(y_j^{(i)},e_2)=0$.
	We can then see that $y_j^{(l)} \perp \{e_1,e_2,\ldots, e_l\}$. This leads to a contradiction.
	Hence, we get the linear system
	\begin{equation}
		0 = P_l(T_{l,l})e_1 = \sum_{j=1}^l \gamma_j P_l(T_{l,l}) y_j^{(l)} = \sum_{j=1}^l \gamma_j P_l(\theta_j^{(l)}) y_j^{(l)}.
	\end{equation}
	Since $\gamma_j$ are nonzero, the linear system determines
        the coefficients of $P_l(t)$ such that the Ritz values
        $\theta_j^{(l)}$ are zeros of $P_l(t)$.
\end{proof}
\begin{remark}
	Note that the objective of the minimization problem
        $f(V_ky_k^*)$ decreases monotonically. The objective of our
        initial guess $f(V_{k+1}[(y_k^*)^T,0]^T)$ is equal to the
        optimal objective of the previous iteration $f(V_ky_k^*)$. At
        worst, it will not make a step towards an improved
        solution. However, this does not imply that we remain stuck
        in this point. Bounds can be removed from the working set
        and the basis is expanded.
\end{remark}

As soon as the subspace spanned by the Lagrange multipliers is
invariant, it is clear which variables are bound by the lower and
upper bounds.  But this subspace might not be large enough to
accurately solve the full KKT system.  If the Lagrange multipliers at
iteration $k>k_0$ can be written as $-\lambda_k + \mu_k = V_{k_0}
d^{(k)}$ for some $d^{(k)} \in \mathbb{R}^{k_0}$,
\cref{eqn:kktP:stationarity} becomes
\begin{equation}
  V_k^T A^T (AV_k y_k -b)+ V_k^T(-\lambda_k + \mu_k) = V_k^T A^T (AV_ky_k -b) +
  V_k^T V_{k_0} d^{(k)} = 0. 
\end{equation}
Depending on the sign of $(V_{k_0} d^{(k)})_i$ for $i \in
\mathcal{W}_k$, the corresponding $x_i = (V_ky_k)_i$ is equal to either
$\ell_i$ or $u_i$.

Let $\mathcal{N} = \{1, \ldots, n \} \setminus \mathcal{A}$ be the
inactive set.  We can then write
\begin{equation}
   \begin{aligned}
     (AV_ky_k)_j &= \sum_{i=1}^n A_{ji}(V_ky_k)_i = \sum_{i\in
       \mathcal{N}} A_{ji}(V_ky_k)_i + \sum_{i\in \mathcal{A}}
     A_{ji}(V_ky_k)_i \\
     &= \sum_{i\in \mathcal{N}} A_{ji}(V_ky_k)_i +
     \sum_{i\in \mathcal{A}, (V_{k_0}d^{(k)})_i <0} A_{ji} l_i + \sum_{i\in
       \mathcal{A}, (V_{k_0}d^{(k)})_i >0} A_{ji} u_i \\
     &= \sum_{i\in   \mathcal{N}} A_{ji}(V_ky_k)_i + \tilde{b}\\
     &= \sum_{i\in   \mathcal{N}} A_{ji}(V_k)_{il}(y_k)_l + \tilde{b},
   \end{aligned}
\end{equation}
where we insert the respective lower and upper bounds in the second equation, depending on the sign of $d^{(k)}$  As a result, their effect can be summarized by a
new right-hand side $\tilde{b}$.  In fact, we have eliminated the
nonlinear complementarity conditions.  Using the notation $\tilde{A}
= {A}_{j \in\{1,2,\ldots, m\},i \in \mathcal{N}}$, we can rewrite the first equation from the KKT
conditions as
\begin{equation}
  \begin{aligned}
    V^T_k A^T \left(\tilde{A} \tilde{V}_k  y_k + \tilde{b}-b\right) + V_k^T V_{k_0} d^{(k)} = 0.
  \end{aligned}
\end{equation}
This corresponds to solving a linear system for $y_k \in \mathbb{R}^k$
and $d^{(k)} \in \mathbb{R}^{k_0}$.  The $y_k$ and corresponding $d_k$
determine the precise value of the Lagrange multipliers within the
subspace $V_{k_0}$

\subsection{Algorithm}
Until now, we have used $y_k \in \mathbb{R}^k$ to denote the solution
of the projected problem \cref{eqn:LSproj}. We now introduce inner
iterations to solve for the optimal $y_k$.  From now on, we use
$y_k^*$ for the optimal solution of subspace $V_k$ and $y_{k}$ for a
guess.  In a similar way, we use $\mathcal{W}_k^*$ to denote the optimal
working set and $\mathcal{W}_k$ for an intermediate guess.

In each outer iteration $k$, the projected optimization problem is
solved with a refined version of the active-set method for convex QP
\cite{nocedal1999numerical}. We call this refined method
QPAS.  The rationale behind the choice of this active-set method is
that it has warm-start capabilities, which greatly improve the
runtime of the algorithm.  This is discussed in detail in
\Cref{sec:warmStart}.

\Cref{alg:resBasis} describes the proposed method. The initial basis
is the normalized initial residual.  We then project the solution to
the subspace and solve the small constrained optimization problem with
a classical active-set method for convex QP.  We start with
$\mathcal{W}_0 = \emptyset$. For $V_1$, the initial guess is $y_1=0$,
and the working set $\mathcal{W}_1=\mathcal{W}_0$ is empty.  After we
have solved the projected problem, we calculate the residual and
expand the basis. From \Cref{thm:ortho}, the basis vector is
orthogonal to the previous ones.

We follow the same strategy as in \cite{vanroose2021krylov}, where a
simplex method solves the minimization of $\|Ax-b\|_1$ and
$\|Ax-b\|_\infty$ by creating a sequence of small projected problems.


\subsection{Stopping criteria}
It is essential to have a stopping criterion to stop the method and
return the solution.  One possibility is to look at the residual
$\|r_k\|_2$ and stop when it is small, see line 8 in
\Cref{alg:resBasis}. Intuitively, this makes sense because $r_k$ is a
measure of the distance to the solution that accounts for the bounds
using the Lagrange multipliers. Alternatively, but closely related,
we can use the loss of positive definiteness in $V^T_kA^TAV_k$. This
occurs when a basis vector is added that is linearly dependent on the
previous basis.
The detection of this loss is relatively straightforward, because a
Cholesky factorization is used and updated (see \Cref{sec:chol} for details). If in our update the last diagonal element of the Cholesky factor becomes (nearly) zero, positive definiteness is lost and the method can be terminated. 

In experiments, both methods seem to work fine, as long as the threshold
is not taken too small (because of rounding errors).
Using both of these methods in tandem is robust in our
experiments, as sometimes semi-convergence would appear. When
semi-convergence sets in, the residual would rise.

\section{Implementation}\label{sec:implementation}
\subsection{Warm-start}\label{sec:warmStart}
By using an active-set method for a convex QP problem for the inner
iterations, we can employ warm-starting.  Let $\mathcal{W}^*_k$ be the
optimal working set, corresponding to the solution $y_k^*$ of subspace
$V_k$, obtained from QPAS (see \Cref{alg:qpas}).  Next, in iteration
$k+1$, we expand the basis to $V_{k+1}$ and warm-start with the
previous solution, i.e., $y_{k+1}=\begin{bmatrix}
  (y_k^*)^T & 0 \end{bmatrix}^T$.  Because of this choice, the
previous working set $\mathcal{W}_k$ is a valid initial working set
for $\mathcal{W}_{k+1}$.

This approach yields a significant runtime improvement.  The method
would otherwise need $|\mathcal{W}_k|$ additional inner iterations for
every $k$. This can be clearly seen in \Cref{fig:warmStart}, where we
compare the number of QPAS inner iterations for each outer iteration
$k$.  It shows that once the active set is more or less discovered and
superlinear Krylov convergence sets in (see \Cref{sec:convergence}),
warm-starting brings significant benefit.

Take special notice of the last 50 iterations, where $\mathcal{W}_k$
remains unchanged. Without warm-starting, $\|\mathcal{W}_k\|$ inner iterations are required to construct $\mathcal{W}_k$ from scratch, for every outer iteration.

\begin{figure}
	\includegraphics[width = \textwidth]{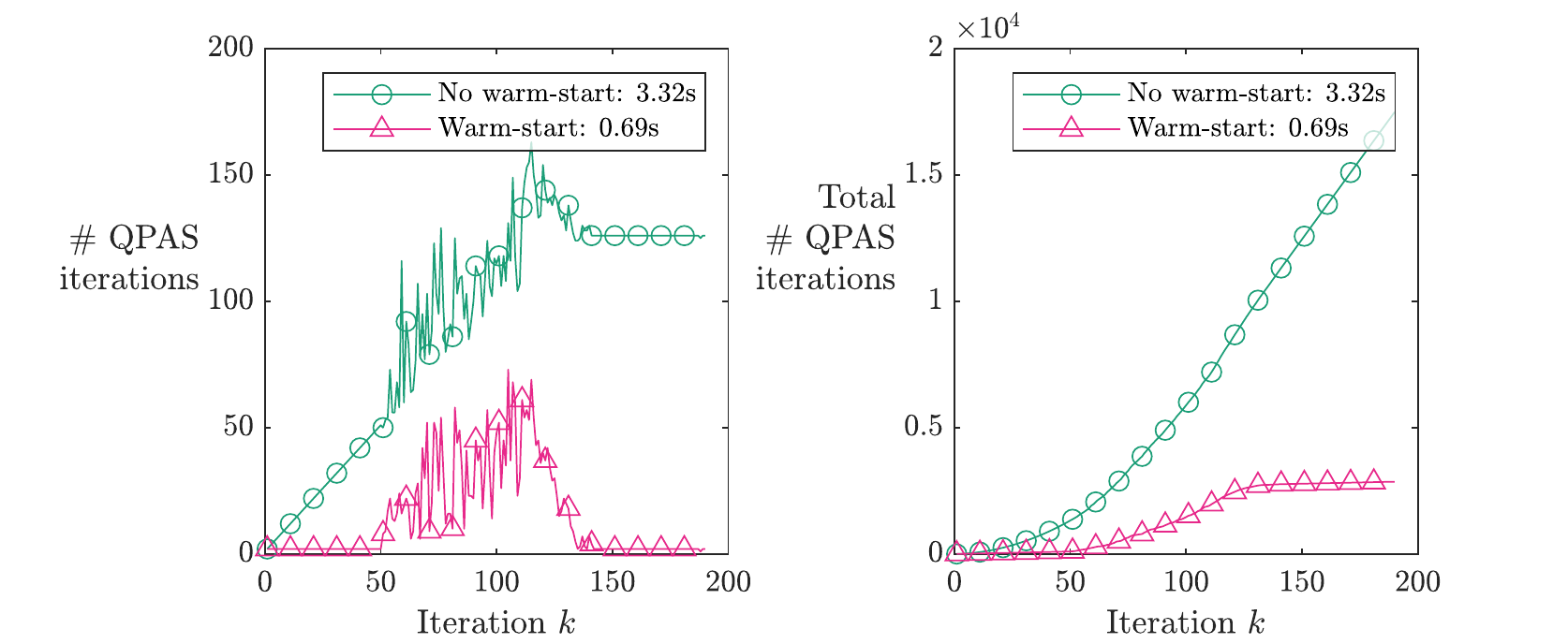}
	\caption{The number of inner iterations required in QPAS
          (\Cref{alg:qpas}) to reach optimality for a given subspace
          of dimension $k$.  When we use warm-starting where the
          working set of the previous optimal subspace is used as
          initial guess, the number of inner iterations is reduced.
        }
	\label{fig:warmStart}
\end{figure}

\subsection{Factorization update}\label{sec:factorization}

\subsubsection{Factorization of the Hessian using Cholesky}\label{sec:chol}
The only linear systems that we need to solve contain the current projected
Hessian $G = V_k^TA^TAV_k \in \mathbb{R}^{k\times k}$. Because it is
symmetric and positive definite, a natural choice is a Cholesky
decomposition. We factor $G$ as $LL^T$, where  $L$ a lower triangular
matrix. As mentioned before, $G$ will only lose its positive
definiteness when the columns of $V$ become linearly dependent (and we
thus have a solution). When this factorization (or its update) fails,
the algorithm can be terminated. This allows us to terminate properly
in cases where the residual does not shrink below the tolerance
(semi-convergence).

When the basis expands, $G$ becomes larger. If $G$ is expanded as
\begin{equation}
	\tilde{G} = \begin{bmatrix} G & c\\ c^T & \delta \end{bmatrix},
\end{equation}
with $c \in \mathbb{R}^k$ a column vector and $d \in \mathbb{R}$ a scalar, the Cholesky factorization can be updated as
\begin{equation} 
	\tilde{L} = \begin{bmatrix} L & 0\\ l^T & \bar{\delta} \end{bmatrix}
\end{equation}
with
\begin{subequations}
	\begin{align}
		Ll & = c,      \\
		\bar{\delta} & = \sqrt{\delta-l^Tl}.
	\end{align}
\end{subequations}

\begin{algorithm}
	\caption{Rank-1 update of lower Cholesky factorization}
	\label{alg:cholesky}
	\begin{algorithmic}[1]
		\Procedure{cholesky\_append}{$L\in \mathbb{R}^{k \times k}$, $c\in \mathbb{R}^k$, $\delta\in \mathbb{R}$}
		\State Solve $Ll = c$ \Comment Forward substitution
		\State $L \gets \begin{bmatrix}L & 0\\ l^T & \sqrt{\delta-l^Tl} \end{bmatrix}$
		\State\Return $L$
		\EndProcedure
	\end{algorithmic}
\end{algorithm}
\begin{remark}
The forward substitution in \lineref{alg:line:forwardBackward} can be
started before the Cholesky update has completed.
\end{remark}

\subsubsection{Factorization of the Hessian using Gram-Schmidt}
The current projected Hessian $G = V_k^TA^TAV_k \in \mathbb{R}^{k\times k}$
can also be factorized using the Gram-Schmidt process. We maintain an orthogonal
basis $U_k \in \mathbb{R}^{m\times k}$ such that $AV_k = U_k B_{kk}$,
where $B_{kk} \in \mathbb{R}^{k \times k}$ is an upper triangular
matrix. This approach is similar to bidiagonalization.  Each time a new normalized 
residual vector $v_{k+1}$ is added to the basis $V_k$, we compute 
$Av_{k+1}$ and make the resulting vector orthogonal to the columns of 
$U_k$. These projections are stored in $B_{kk}$.  Consequently, we have
\begin{equation}
  V_k^TA^TAV_k = B_{kk}^TU_k^T U_k B_{kk} = B_{kk}^T B_{kk}.
\end{equation}
This represents the Hessian as a product of a lower and upper
triangular matrix.  We observe that the matrix
$B_{kk}$ becomes a bidiagonal matrix once Krylov convergence sets
in. At this moment we do not exploit this bidiagonal asymptotic
structure. If we have a robust way to detect when the asymptotic phase
sets in it can be exploited and the link with Golub-Kahan
bidiagonalization can be made.  \Cref{alg:outer} and
\Cref{alg:outer_gs} show the outer loops with Cholesky and modified
Gram-Schmidt factorization, respectively.  Note that the $B_{kk}$, from
this section, and the Cholesky factor from \Cref{sec:chol} are each
others transpose.

\subsubsection{Solving the linear system in QPAS with QR}\label{sec:qr}
In the active-set method (see \Cref{alg:qpas}) a convex QP problem of the form
\begin{equation}\label{eqn:qp}
\min_x x^TGx + f^Tx \quad \text{s.t.} \quad Cx \leq d 
\end{equation}
is solved, where $G\in \mathbb{R}^{k\times k}$ is the convex Hessian,
$f\in \mathbb{R}^k$ the linear cost term, and $C \in
\mathbb{R}^{t\times k}$ is the matrix representing inequality
constraints.  We have limited ourselves to the least squares problem,
meaning $G$ is positive definite, as mentioned in \Cref{sec:chol}.


The algorithm solves a linear system of equations each iteration to
obtain the search direction and the Lagrange multipliers $\lambda_k$.
The linear system is given by
\begin{equation}
	\begin{pmatrix}
		G                 & C^T_{\mathcal{W}_k} \\
		C_{\mathcal{W}_k} & 0
	\end{pmatrix}
	\begin{pmatrix}
		-p_k \\
		\lambda_k
	\end{pmatrix}
	= \begin{pmatrix}
		Gx_k+f \\
		0
	\end{pmatrix},
\end{equation}
where $C_{\mathcal{W}_k}$ is the matrix with rows of $C$ from
$\mathcal{W}_k$, and $x_k$ is the current approximation of \Cref{eqn:qp}
as defined in \Cref{alg:qpas}.  Block elimination leads to
\begin{equation}
	C_{\mathcal{W}_k}G^{-1}C^T_{\mathcal{W}_k}\lambda_k = C_{\mathcal{W}_k}(x_k + G^{-1}f) =:q.\label{eqn:lambdaSystem}
\end{equation}
Additionally, from \Cref{sec:chol} we have a Cholesky-factorization of
$G=LL^T$. By computing the QR-factorization
\begin{equation} \label{eq:qr}
QR = L^{-1}C^T_{\mathcal{W}_k},
\end{equation}
we can rewrite \Cref{eqn:lambdaSystem} as
\begin{equation}
R^TR\lambda_k = C_{\mathcal{W}_k}(x_k + G^{-1}f).
\end{equation}

During each iteration, the working set $\mathcal{W}_k$ is expanded by
one index, reduced by one index, or remains unchanged. The QR
factorization of $L^{-1}C^T_{\mathcal{W}_k}$ is updated
accordingly with routines that add or remove
rows and columns.
\cite{hammarling2008updating}.

\subsection{Limiting the inner iterations}\label{sec:limittingInner}
In the early iterations, it is often better to expand the basis first
and then find the optimal coefficients. Indeed, a small basis cannot
fully represent the solution.  This suggests to limit the number of
inner iterations and calculate the residual when this upper bound is
reached. In principle, we only require optimality at the final
subspace.  Limiting the number of inner iterations requires care. If
we stop the inner iteration without achieving feasibility, the
orthogonality between the residuals is lost, see
\Cref{thm:ortho_feasibility}.

As shown in \Cref{fig:limitInner}, which shows the number of inner
iterations required to find the optimal solution for subspace $V_k$,
there is a sweet spot for the maximal number of inner iterations.  If
the number of inner iterations is high, we find the optimal solution
in each intermediate subspace and we converge with the smallest basis.
However, each outer iteration takes longer because of the large number
of inner iterations.  Especially at the halfway point, it appears to
be not useful to solve the projected problem exactly. An approximate
(feasible) solution is sufficient up to a point.  If the number of
inner iterations is limited too harshly, additional basis vectors may
be needed, leading to a larger subspace, which slows the convergence of
the residuals. In this example, the optimal number of inner iterations
appears to be 5, resulting in a 33\% speed-up over not limiting the inner iterations.

Note that there is no effect on the orthogonality of the residuals if
we stop with a feasible solution. This can be guaranteed if we stop
the inner QPAS iterations when the maximum number of inner iteration
is reached \textit{and} the QPAS search direction has zero length
(i.e., $p_k=0$ in \Cref{alg:qpas}).

\begin{figure}
	\includegraphics[width = \textwidth]{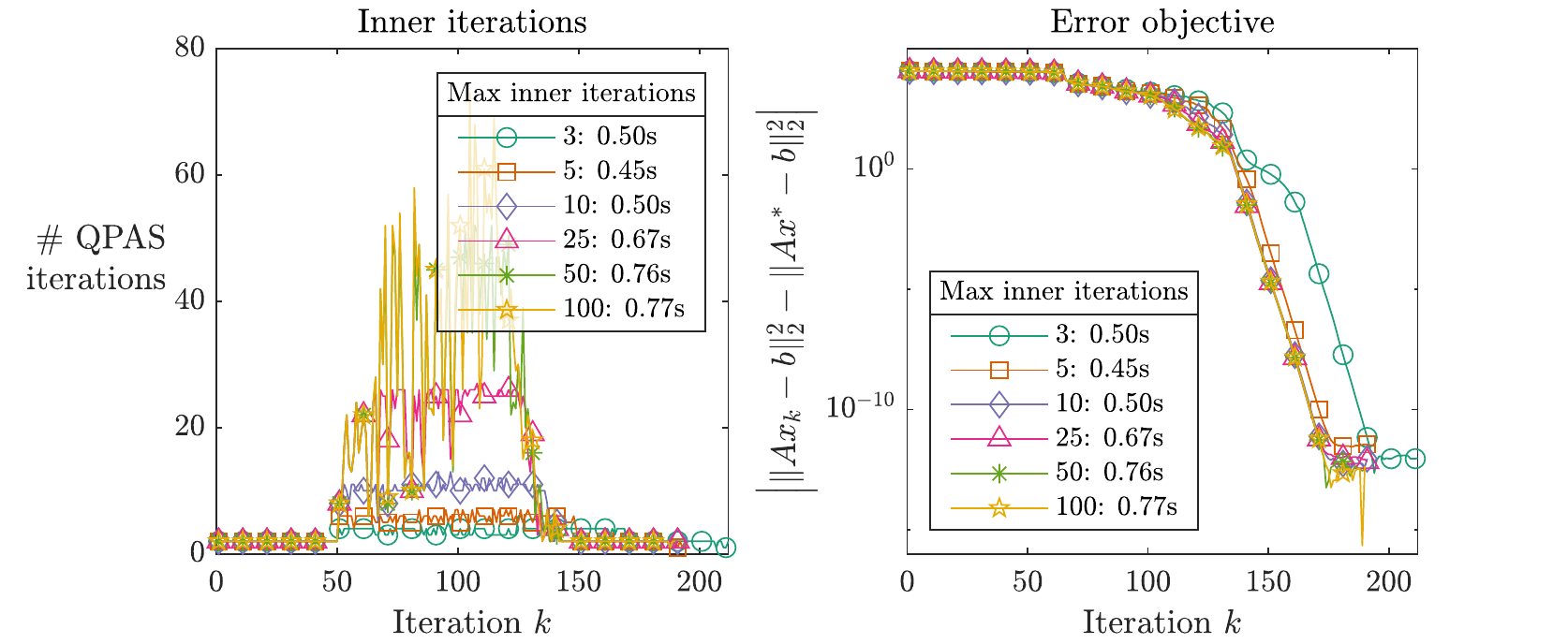}
	\caption{Limiting the number of inner QPAS iterations in
          \Cref{alg:resBasis} has little effect on convergence but
          a significant effect on the time to solution.  At the left, we show
          the number of inner iterations for various choices of the
          upper limit. On the right, we display the effect on 
          convergence.  The time to solution is shown for each choice
          of the upper limit.
	\label{fig:limitInner}}
\end{figure}


The final algorithm with all improvements is written in \Cref{alg:outer} and \Cref{alg:qpas}

\begin{algorithm}
	\caption{Outer loop with Cholesky}
	\label{alg:outer}
	\begin{algorithmic}[1]
		\Procedure{ResQPASS}{$A$, $b$, $\ell$, $u$, $\epsilon_1$, $\epsilon_2$, {\tt maxit}}
		\State $r_0 = -A^Tb$
		\State $V_1 = \begin{bmatrix} r_0/\|r_0\| \end{bmatrix}$
		\State $L = \|AV_1\|$ \Comment{\Call{lower\_cholesky}{$V_1^TA^TAV_1$}}
		\State $\mathcal{W}_0 = \emptyset$
		\State $y_0 = []$
		\For{$i = 1,2,\ldots$}
		\State $\left[y_i,\mathcal{W}_i,[\lambda_i;\mu_i]\right] \gets$ \\ \hspace{2cm}\Call{qpas}{$L, -b^TAV_i, [-V^T, V^T]^T, [\ell^T, u^T]^T, [y_{i-1}^T, 0]^T, \mathcal{W}_{i-1}$, {\tt maxit}}
		\State $r_i = A^T(AV_iy_i-b) - \lambda_i + \mu_i$
		\State $v_{i+1} = r_i/\|r_i\|$
		\State $V_{i+1} = \begin{bmatrix} V_i & v_{i+1}\end{bmatrix}$
		\If{$\|r_i\|<\epsilon_1$}
		\State\Return $V_iy_i$
		\EndIf
		\State $c \gets V_i^TA^TAv_{i+1}$
		\State $d \gets v_{i+1}^TA^TAv_{i+1}$
		\State $L\gets$ \Call{cholesky\_append}{$L,c,d$} \label{alg:outer:cholAppend} using algorithm \ref{alg:cholesky}
		\If{$|L|<\epsilon_2$}\Comment{$V_i^TA^TAV_i$ no longer pos. def.}
		\State\Return $V_iy_i$
		\EndIf
		\EndFor
		\EndProcedure
	\end{algorithmic}
\end{algorithm}

\begin{algorithm}
	\caption{Outer loop with modified Gram-Schmidt}
	\label{alg:outer_gs}
	\begin{algorithmic}[1]
		\Procedure{ResQPASS}{$A$, $b$, $\ell$, $u$, $\epsilon_1$, $\epsilon_2$, {\tt maxit}}
		\State $r_0 = -A^Tb$
		\State $V_1 = \begin{bmatrix} r_0/\|r_0\| \end{bmatrix}$
		\State $\mathcal{W}_0 = \emptyset$
		\State $y_0 = []$
		\For{$i = 1,2,\ldots$}
                    \State $u_{i} = Av_{i}$
                    \For{$k=1,2,\ldots, i\!-\!1$}
                       \State $B_{ki} = u_k^Tu_i$
                       \State $u_i = u_i- u_k B_{ki}$
                    \EndFor
                    \State $B_{ii} = \|u_i\|$
                    \If {$i=1$}
                    \State $U_1 = \begin{bmatrix} u_i/B_{ii} \end{bmatrix}$
                    \Else
                    \State $U_{i} = \begin{bmatrix} U_{i-1} & u_{i}/B_{ii} \end{bmatrix}$
                    \EndIf
		    \If{$|B_{ii}|<\epsilon_2$}\Comment{$V_i^TA^TAV_i$ no longer pos. def.}
		       \State\Return $V_iy_i$
		    \EndIf

		\State $\left[y_i,\mathcal{W}_i,[\lambda_i;\mu_i]\right] \gets$ \\ \hspace{2cm}\Call{qpas}{$B^T, -b^TAV_i, [-V^T, V^T]^T, [\ell^T, u^T]^T, [y_{i-1}^T, 0]^T, \mathcal{W}_{i-1}$, {\tt maxit}}
		\State $r_i = A^T(AV_iy_i-b) - \lambda_i + \mu_i$
		\State $v_{i+1} = r_i/\|r_i\|$
		\State $V_{i+1} = \begin{bmatrix} V_i & v_{i+1}\end{bmatrix}$
		\If{$\|r_i\|<\epsilon_1$}
		\State\Return $V_iy_i$
		\EndIf
		\EndFor
		\EndProcedure
	\end{algorithmic}
\end{algorithm}

\begin{algorithm}
	\caption{quadratic programming active-set (QPAS) with lower triangular factor of Hessian and QR updates 
}
	\label{alg:qpas}
	\begin{algorithmic}[1]
		\Procedure{qpas}{$L,f,C,d,x_0,\mathcal{W}_0, \tt{maxit}$} 
		\State $\mathcal{I} = \{1,\ldots,t\}$
		\State\label{alg:line:forwardBackward}Solve $LL^Tz=f$ \Comment{$z:=G^{-1}f$}
		\For{$k=0,1,2\ldots$}
		\State {Compute $p_k,\lambda_k$ and helper variables for equality constrained subproblem:}
		\If{$|\mathcal{W}_k|=0$}
		\State $p_k \gets -(x_k + z)$    \Comment{$p_k= - \left( x_k+\left(G^{-1}f\right)\right)$}
		\State $\lambda, X,Y, q \gets \emptyset$
		\Else
		\If{$k=1$} \Comment{Warm-start}
		\State Solve: $LX=C^T_{\mathcal{W}_k}$     \Comment{$X:=L^{-1}C^T_{\mathcal{W}_k}$}
		\State Solve: $L^TY=X$ \Comment{$Y:=G^{-1}C^T_{\mathcal{W}_k}$}
		\State $[Q,R] \gets \textsc{qr}\left(X\right)$
		\State\label{alg:line:qpas_q} $q \gets C_{\mathcal{W}_k}\left( x_k + z \right)$
		\EndIf
		\State Solve: $R^TR\lambda_k = q$  \Comment{$C_{\mathcal{W}_k}G^{-1}C^T_{\mathcal{W}_k}\lambda_k = C_{\mathcal{W}_k}(x_k + G^{-1}f)$}
		\State $p_k \gets - \left( x_k+z+ Y\lambda_k \right)$ \Comment{$p_k= - \left( x_k+\left(G^{-1}f\right)+ \left(G^{-1}C^T_{\mathcal{W}_k}\lambda_k \right) \right)$}
		\EndIf
		\State {Compute $x_{k+1}$ and update the working set $\mathcal{W}_{k+1}$:}
		\If{$p_k=0$} \Comment{Optimum reached or unnecessary bound}
		\If{$\forall i \in \mathcal{W}_k\cap\mathcal{I}: (\lambda_k)_i \geq 0 \textbf{ or } k\geq \tt{maxit}$}
                \State \Return $[x_k,\mathcal{W}_k,\lambda_k]$
		\Else
		\State $j\gets \argmin_{j\in\mathcal{W}_k\cap\mathcal{I}}\lambda_j$
		\State $\mathcal{W}_{k+1}\gets\mathcal{W}_k\setminus\{j\}$
		\State $ X \gets$ remove column $j$ in $X$  \Comment{Downdate $L^{-1}C^T_{\mathcal{W}_{k+1}}$}
		\State $ Y\gets$ remove column $j$ in $Y$   \Comment{Downdate $G^{-1}C^T_{\mathcal{W}_{k+1}}$}
		\State $[Q,R] \gets$ Update QR factors of $X$ with $j$ removed
		\State $q\gets$ Remove $q_j$ from $q$
		\State $x_{k+1}\gets x_k$
		\EndIf
		\Else \Comment{$\exists$ new blocking constraint}
		\State $\alpha_k \gets \min\left(1,\min_{i\not\in\mathcal{W}_k,\ C_ip_k<0}\frac{d_i-C_ix_k}{C_ip_k}\label{alg:line:alphak}
			\right)$
		\State $x_{k+1} \gets x_k + \alpha_kp_k$
		\If{$\exists j \in \mathcal{I}\setminus\mathcal{W}_k$ blocking constraint}
		\State $\mathcal{W}_{k+1} \gets \mathcal{W}_k \cup \{j\}$
		\State  $X \gets \begin{pmatrix} X & \delta x \end{pmatrix}$ with $L\delta x=C^T_j$  \Comment{Expand to $L^{-1}C^T_{\mathcal{W}_{k+1}}$}
		\State  $Y \gets \begin{pmatrix} Y & \delta y \end{pmatrix}$ with $L^T\!\delta y=\delta x$ \Comment{Expand to $G^{-1}C^T_{\mathcal{W}_{k+1}}$}
		\State $[Q,R]\gets$ Update QR factors of $X$ with column $L^{-1}C^T_j$ added
		\Else
		\State $\mathcal{W}_{k+1}\gets\mathcal{W}_k$
		\EndIf
		\EndIf
		\EndFor
		\EndProcedure
	\end{algorithmic}
\end{algorithm}

\subsection{Replacing  matrix-vector multiplication with a recurrence relation in QPAS}
When a nonzero step $p_k$ is computed, a step size $\alpha_k$ is
determined. The definition of $\alpha_k$ can be found in
\lineref{alg:line:alphak}, where $C_i$ is the $i$-th row of matrix
$C$.  Two matrix-vector products are computed in this step: $Cx_k$ and
$Cp_k$. However, the definition of $x_{k+1}=x_k+\alpha_kp_k$ implies
the possibility of a recursive relationship for $Cx_{k+1}$:
\begin{equation}
Cx_{k+1} = Cx_k + \alpha_k Cp_k.
\end{equation}
Here, $Cp_k$ was already computed to find $\alpha_k$, where it appears
in the denominator. We expect that using recursion reduces runtime, as
it avoids a matrix-vector product. This is similar to
CG \cite{liesenstrakos} and Pipelined CG \cite{ghysels2014hiding}.  As
in CG, it introduces additional rounding errors that have a slight
delay effect on the convergence \cite{cools2018analyzing}.

In \Cref{fig:timingInstable}, we compare the recursive case with 5 and
10 inner iterations to the direct case for a problem from
\Cref{ex:tuneable} with $A\in\mathbb{R}^{10,000\times 6,000}$ and
$i_{\max}=256$. The norm of the recursion error is also shown. As expected, the
direct method is the slowest. The case with 10 inner
iterations behaves exactly the same as the direct case, only 5\%
faster. The case with 5 inner iterations is slightly slower because of
the delayed convergence caused by the stringent bound on the number of inner
iterations; however, it is still faster than the direct method.

\begin{figure}
\includegraphics[width=\textwidth]{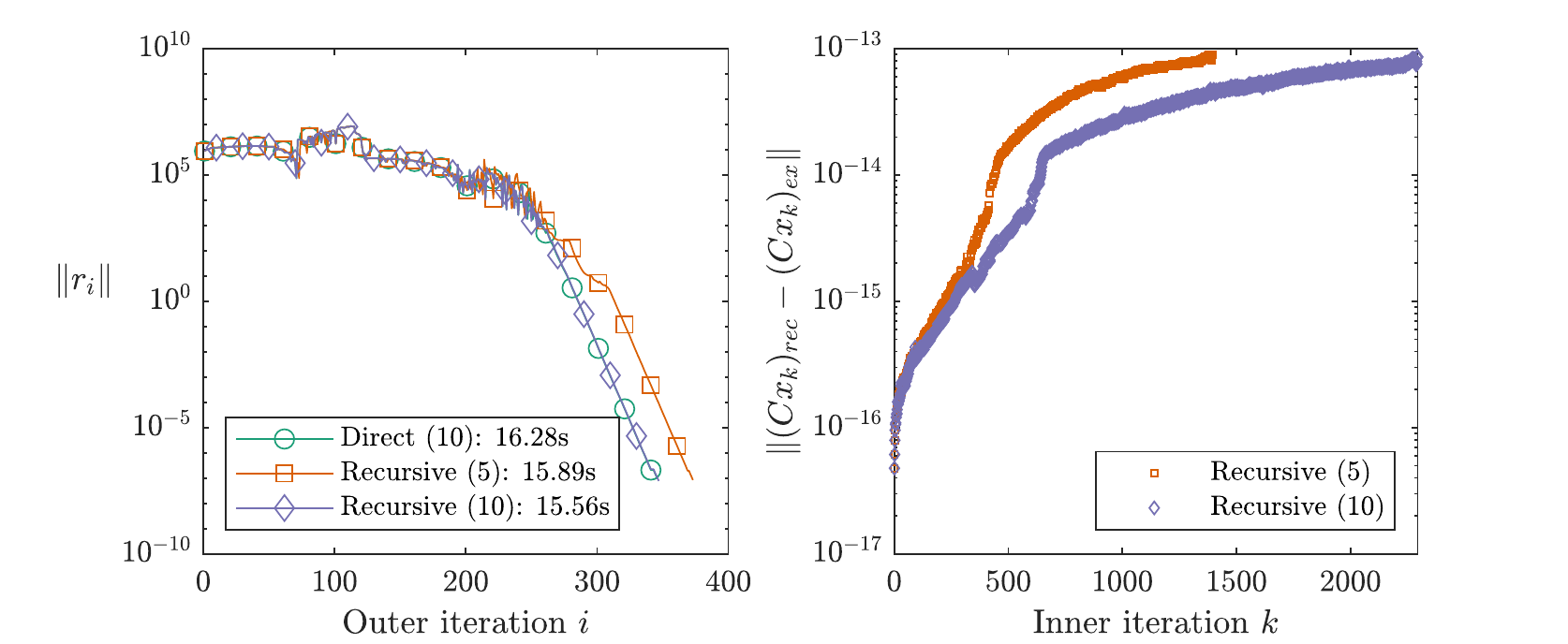}
\caption{ Comparison between the recursive version with an upper limit
  of 5 and 10 inner iterations and the direct version with 10 inner
  iterations. The difference between the exact solution $(Cx_k)_{ex}$
  and the recursively computed $(Cx_k)_{rec}$ remains low
  ($<10^{-13}$) during this experiment. There is no significant effect
  on the number of outer iterations but there is a measurable
  improvement in runtime.}
\label{fig:timingInstable}
\end{figure}

Note that a similar recurrence is used in CG for the residuals. It also
avoids a sparse-matrix multiplication, and its effect on the rounding
error is well studied, see, for example, Chapter 5 in \cite{liesenstrakos}.




\section{Numerical experiments}\label{sec:numerical}
Our research is academic in nature and mainly explores the
possibilities that projections and Krylov-like methods still
offer. Therefore the problems are synthetic for now, to show that
ResQPASS is promising. Detailed error analysis and final attainable
accuracy are future work along with in-depth studies of specialised
applications.

\subsection{LSQR}
As mentioned in \Cref{ex:cg}, the unconstrained ResQPASS applied to
$\|Ax-b\|$ converges like LSQR, which is equivalent to CG applied to
$A^TAx=A^Tb$. For this experiment, the entries of
$A\in\mathbb{R}^{m\times n}$ and $x_{ex}\in\mathbb{R}^{n}$ are drawn
from a normal distribution with mean 0 and standard deviation 1. All
values smaller than 0 and larger than 0.1 are set to 0, resulting in a
sparse $A$ with $3.98\%$ fill. The right-hand side is defined as
$b=Ax_{ex}$. In our experiment, we consider both overdetermined and
underdetermined systems with an aspect ratio of 10:6. Note that the
matrix $A^TA$ is nearly a full matrix, particularly for large $A$. In
\Cref{fig:timingCG} we see that ResQPASS is slower than LSQR, which is
expected, as ResQPASS is more general and accounts for the bounds on
$x$. The small increase in runtime between LSQR and ResQPASS can be
interpreted as the overhead required for this generalization.

\begin{figure}
	\includegraphics[width=\textwidth]{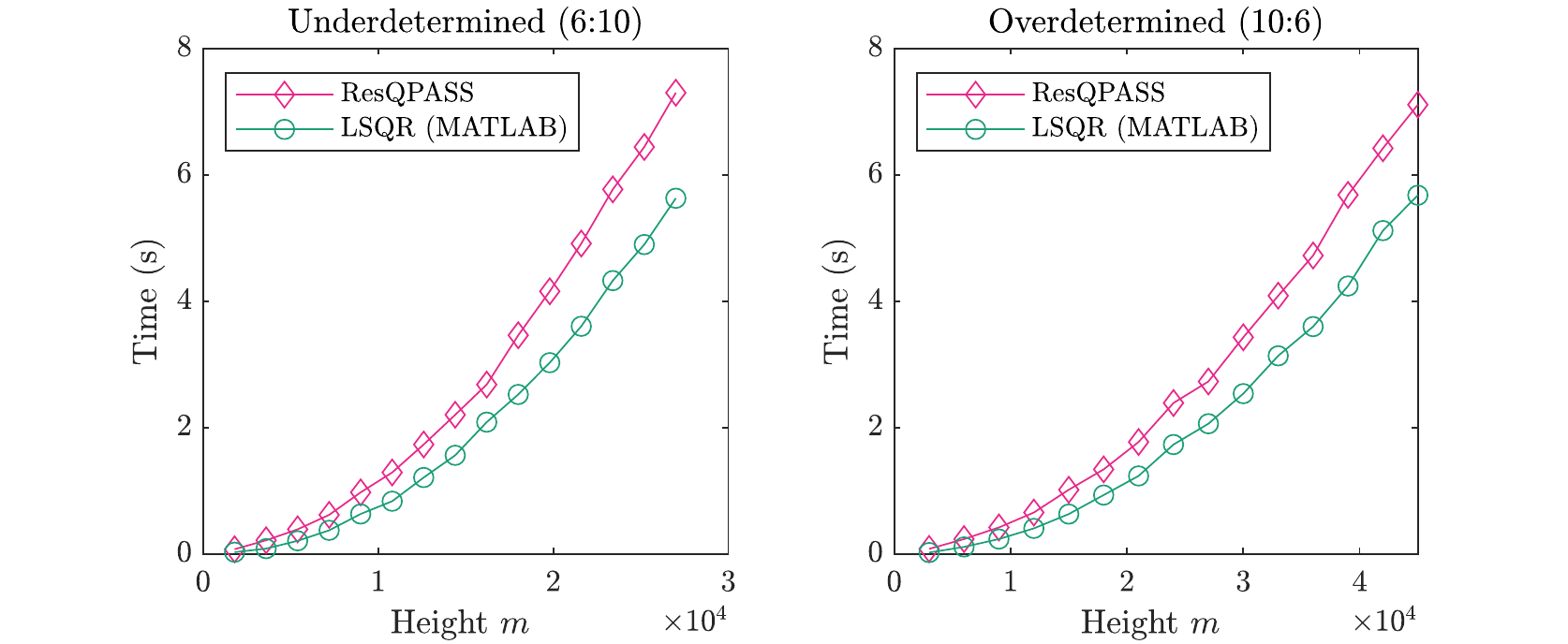}
	\caption{Comparison in runtime for varying sizes of $A$
          between ResQPASS and MATLAB's LSQR implementation
          \texttt{lsqr}. The difference in runtime can be interpreted
          as the overhead of dealing with bounds.
        }
\label{fig:timingCG}
\end{figure}

\subsection{Bounded Variable Least Squares (BVLS)}
In this experiment, we use the formulation of \Cref{ex:tuneable} with
$A\in \mathbb{R}^{10,000\times 6,000}$. The numerical improvements
made to the QPAS method already gives a significant advantage compared
to MATLAB's active-set solver and a basic QPAS
implementation.

Projection still appears to be beneficial, as ResQPASS is faster than
using QPAS from \Cref{alg:qpas} without projections.  In
\Cref{fig:timingTuneable}, it is evident that for this problem, which
involves a small number of active constraints, ResQPASS outperforms
the interior-point method (IP).  However, for sufficiently large
$i_{\max}$, ResQPASS would eventually be surpassed by IP (as is the case
with MATLAB's \texttt{quadprog} with \texttt{'active-set'}).  This is
because the runtime of interior-point methods is known to be
independent on the number of active constraints, whereas for ResQPASS
(and other active-set algorithms), it is dependent on the number of
active constraints \cite{nocedal1999numerical}.

\begin{figure}
\includegraphics[width = \textwidth]{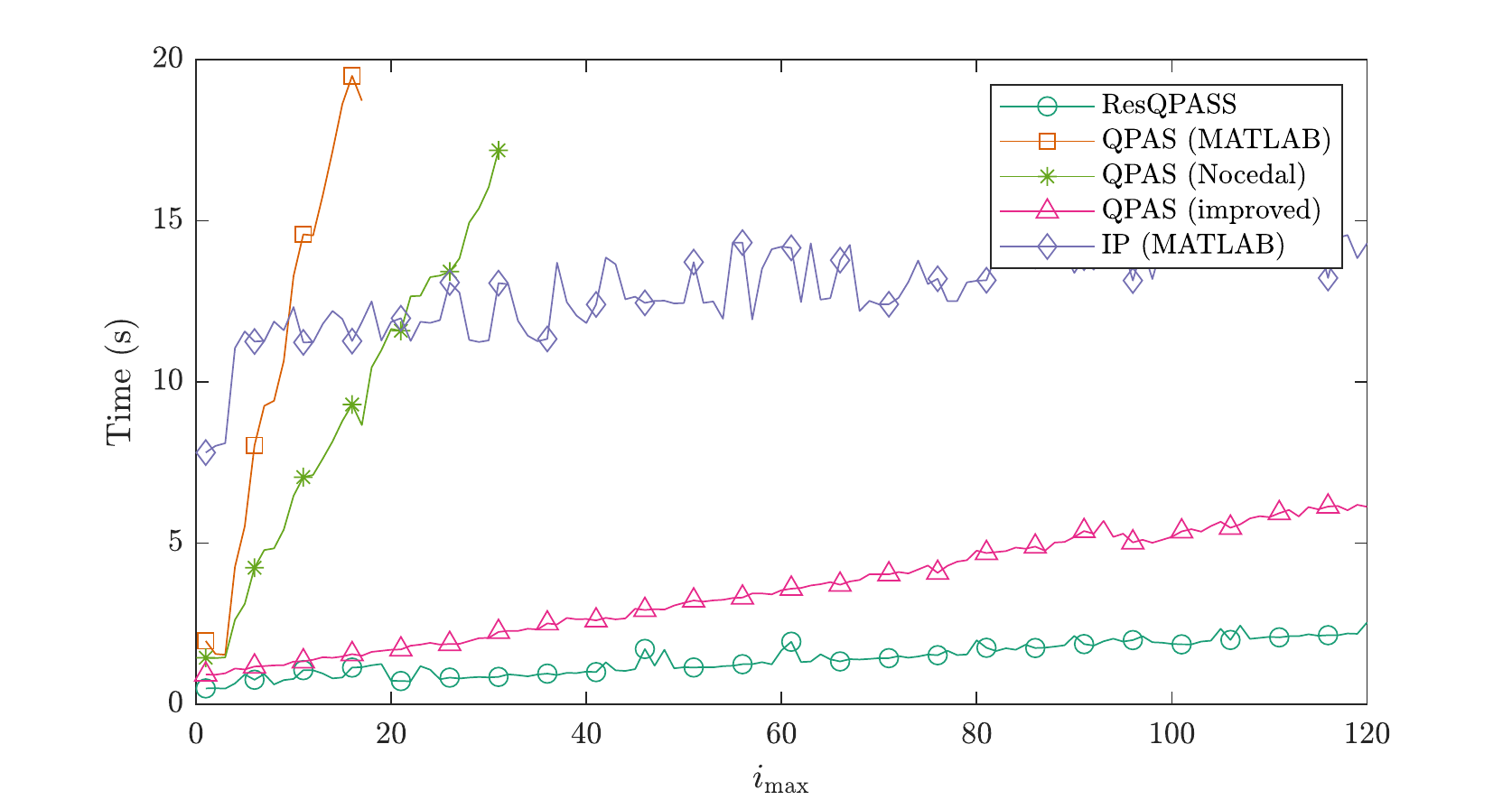}
\caption{Comparison in runtime between ResQPASS, MATLAB's
  \texttt{quadprog} with \texttt{'active-set'} implementation, the
  basic QPAS implementation described in \cite{nocedal1999numerical},
  \Cref{alg:qpas}, and MATLAB's interior point implementation (default \texttt{quadprog}) for
  different $i_{\max}$.}
\label{fig:timingTuneable}
\end{figure}

\subsection{Nonnegative Matrix Factorization (NMF)}
The $n\times m$ matrix $A$ is given, and we aim to find the unknown
matrices, with a given rank $p$, $X\in\mathbb{R}^{n\times p}$ and
$Y\in\mathbb{R}^{p\times m}$, with $X,Y\geq 0$ (elementwise), such
that $\|A-XY\|$ is minimized. This can be achieved with the
Alternating Least Squares (ALS) algorithm \cite{hastie2015matrix}. The
ALS algorithm requires solving nonnegative least squares problems, for
which ResQPASS can be applied. As mentioned in \Cref{rem:shift}, the
problem should be shifted such that 0 is feasible, as QPAS needs a
feasible initial guess.

Let the function $\VEC(X)$ be the vectorization of the matrix $X$,
where columns are stacked. We can rewrite $\|A-XY\|^2$ as
$\|\VEC(A)-(I_m\otimes X)\VEC(Y)\|^2$ and also as  $\|\VEC(A) - (Y^T\otimes
I_n)\VEC(X)\|^2$ to obtain a classical least squares formulation.

\begin{algorithm}
	\caption{NMF with ALS}
	\label{alg:ALS}
	\begin{algorithmic}[1]
		\Procedure{ALS}{A,p}
			\State $X = \text{rand}(n,p)$
			\For{$i\in\{1,\ldots,10\}$}
				\State $Y = {\displaystyle\argmin_{Y\geq 0}} \frac{1}{2}\|A-XY\|^2$
				\State $X = {\displaystyle\argmin_{X\geq 0}} \frac{1}{2}\|A-XY\|^2$
			\EndFor
		\EndProcedure
	\end{algorithmic}
\end{algorithm}

We use the synthetic example of \cite{gu2023fast}, where the matrix $A$
is constructed as follows. $X,Y$ are $n\times p$ and $p\times m$
matrices, respectively, with uniformly drawn random entries between 0
and 1. Gaussian noise with mean 0 and standard deviation 0.1 is added
to the product $XY$, and all negative values are set to 0 to obtain
$A$.

The number of ALS iterations is fixed to 10 (20 minimizations), and the
maximum number of inner QPAS iterations for ResQPASS is set to 10.  We
use a naive approach where the entire system is minimized at once
using either MATLAB's built-in nonnegative least squares solver
(\texttt{lsqnonneg)} or ResQPASS.  See \cite{gillis2020nonnegative}
and \cite{kim2011fast} for state-of-the-art implementations.  In
\Cref{fig:timingALS} we observe that ResQPASS offers a significant speed
advantage (one or two orders of magnitude) over \texttt{lsqnonneg},
  without sacrificing accuracy of the solution, for the naive approach.

\begin{figure}
\includegraphics[width = \textwidth,trim={1cm 8mm 1cm
    0},clip]{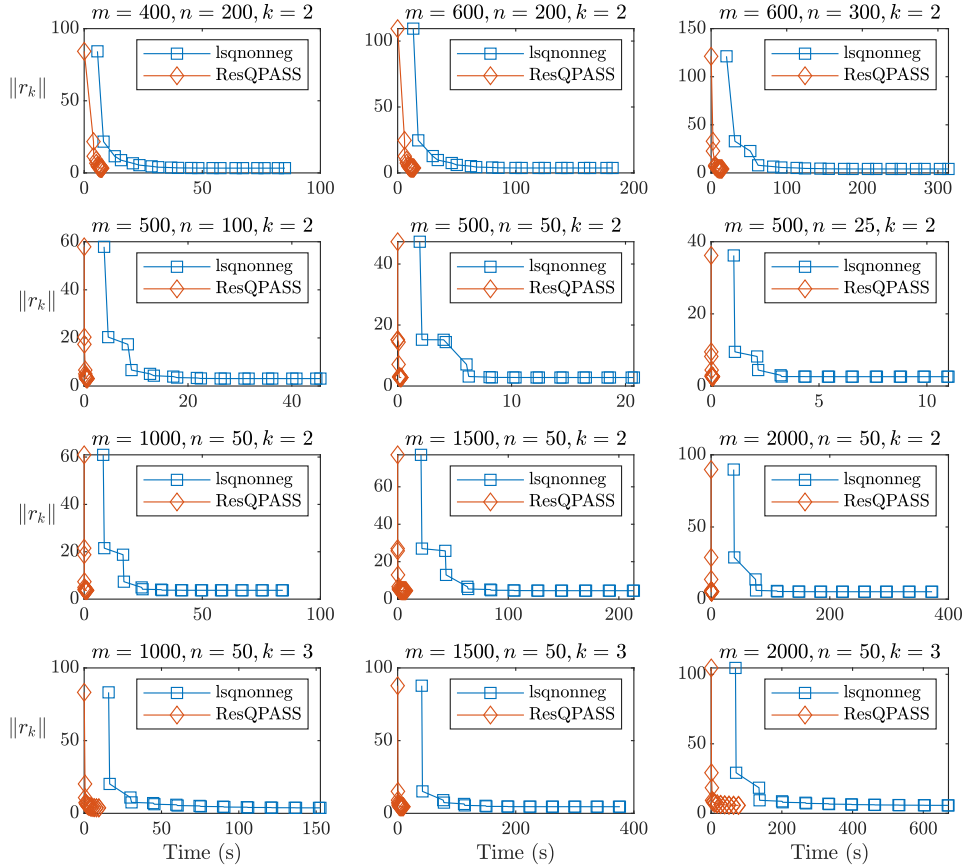}
\caption{Comparison between ALS implemented with either MATLAB's
  nonnegative least squares (\texttt{lsqnonneg}) or ResQPASS for varying sizes
  and 10 ALS iterations.}
\label{fig:timingALS}
\end{figure}

\subsection{Contact problem}  \label{example:contact}
In contact problems, the solution is constrained by spatial
limitations. Bounds restrict the solution space and we require a
subspace that is adapted to these conditions.  We study the
convergence of a simplified model for an inflated balloon with limited
space for inflation. In a balloon, the elastic forces are balanced by
the internal pressure. This creates an equilibrium between the
curvature and the magnitude of the pressure, meaning that the
Laplacian applied to the solution equals the pressure vector, which
remains constant throughout the domain. This leads to a system of
equations $Lx = p$, where $L \in\mathbb{R}^{n \times n}$ is the
Laplacian, $x \in \mathbb{R}^n$ is the solution that models the
surface deformation of the balloon, and $p \in \mathbb{R}^n$ is the
pressure vector.

If the solution is limited by lower and upper bounds, we have a bounded
variables QP problem. In this case we minimize the energy with bounds
on the solution:
\begin{equation} \label{eqn:contact_problem}
  \begin{aligned}
    \min_x \| L x - p\|^2_2 \quad s.t. \quad \ell \le x \le u.
  \end{aligned}
\end{equation}
For the simplified example, we use $\ell=0$ and $u=0.1$ and a constant
pressure vector.  The domain is $[0,1]^2$.

The matrix is the 3-point stencil approximation of the Laplacian. It
can be written as $A_\text{2D} = A_{1D} \otimes I + I \otimes A_{1D}
$, where $A_{1D}$ is the 1D finite-difference Laplacian matrix, a
tridiagonal matrix with constant diagonals.  We use a mesh with 50x50
grid points and a pressure vector with a constant value of 4.

In \Cref{fig:convergence_contact}, we show the convergence of ResQPASS
for the contact problem, both with and without preconditioning. The
solutions after 200 iterations are shown in
\Cref{fig:solution_contact}.

The asymptotic Krylov convergence is determined by the spectral
properties of $A^TA$.  It can be improved by an operator $M$ that is an
approximate inverse that is cheap to invert.  To illustrate this,
we use a preconditioning matrix that is based on an incomplete LU
factorization of the Hessian $A^TA$. We use a threshold of 0.1 for the
{\tt ilu} factorization.  Instead of using the residual \cref{eqn:rk}
we use a preconditioned residual, the solution of 
\begin{equation}
   M r_k =  A^T(AV_ky_k-b) - \lambda_k + \mu_k,
\end{equation}
where $M = L U$ is from the incomplete factorization.  When none
of the bounds is active, this corresponds to ILU preconditioned
CG.

The unpreconditioned iteration does not reach the bounds in the first
200 iterations and follows the CG convergence.  However, for the
preconditioned solution, we reach the bounds after 15 iterations.
From then on, the subspace is adapted to the bounds.  The
solution is adapted to these bounds, see left-hand side of
\Cref{fig:solution_contact}.

A more in depth study of various preconditioning techniques for
ResQPASS is left as future work.

\begin{figure}
	\includegraphics[width=\textwidth]{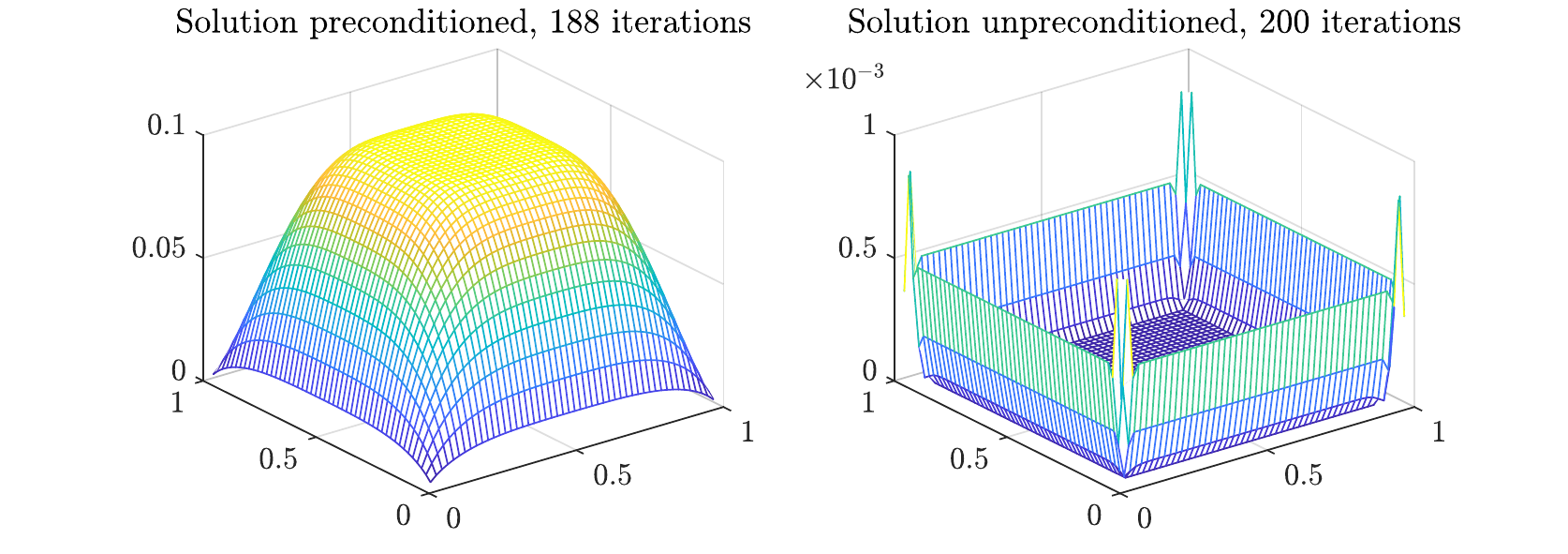}
  \caption{Solution after 200 iterations for the preconditioned and
    unpreconditioned problem \Cref{eqn:contact_problem}.  The
    unpreconditioned iteration converges very slowly and has not
    reached the bounds by iteration 200. However, the preconditioned
    iteration reaches the bounds starting from iteration 15. As
    explained earlier, the residuals that are used as a basis then
    start to include effects from the Lagrange multipliers and the
    solution adapts to the bounds. The preconditioned version reaches
    the desired tolerance before the maximum number of iterations is
    reached.}
\label{fig:solution_contact}
\end{figure}

\begin{figure}
\includegraphics[width=\textwidth]{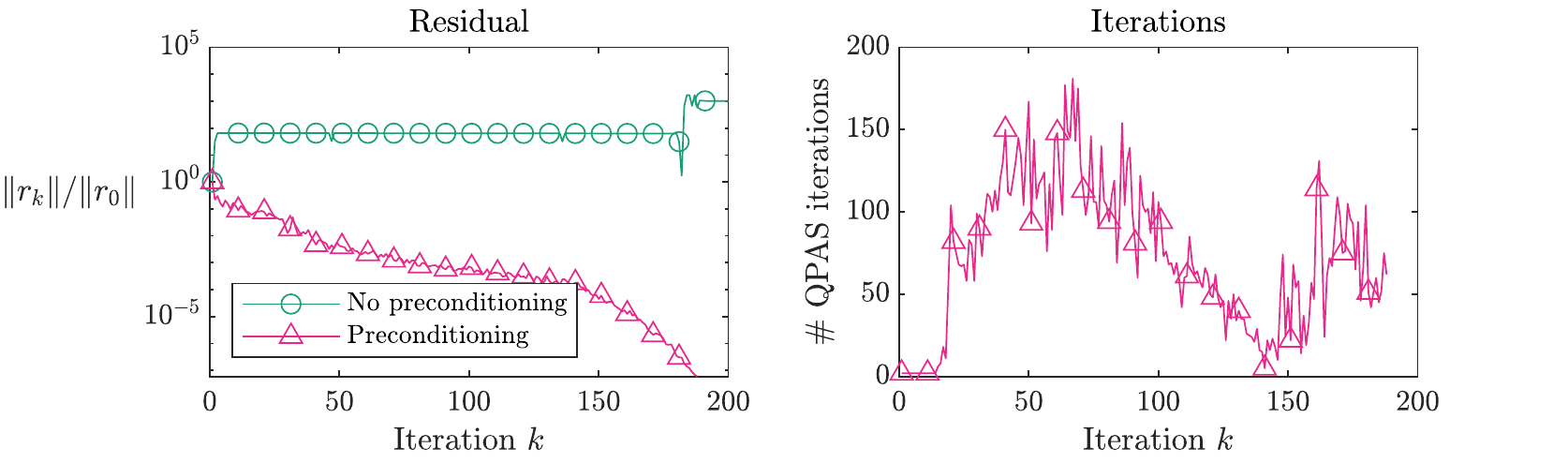}
  \caption{Left: The convergence for the contact problem
    \cref{eqn:contact_problem} for the ResQPASS algorithm, with and
    without ILU preconditioning for $A^T A$. Preconditioning
    significantly accelerates the convergence of the solver.\\ Right:
    We show the number of inner QPAS iterations for the preconditioned
    problem.  The first few iterations no bounds are reached, hence no
    inner iterations are required.  But then as the bounds become
    active a significant number of inner iterations is used.  Because
    of the preconditioning, the solution can change significantly from
    one solution to the next, hence the significant number of inner
    iterations.  }
\label{fig:convergence_contact}
\end{figure}
%

\section{Discussion and conclusion}\label{sec:conclusion}
We have presented ResQPASS, the \textit{residual quadratic programming active-set
  subspace method} for solving box-constrained linear least
squares problems with sparse matrices.  In this work, we propose the
subspace method, analyze the convergence and detail its efficient
implementation with warm-starting and factorization updates.

Similar to Krylov methods, the sequence of generated residual vectors is 
pairwise orthogonal. The residual vectors we use include the current
approximated Lagrange multipliers. The inner problem, a small QP
problem with dense matrices, only needs to be solved for feasibility
to obtain an orthogonal basis.

ResQPASS works quickly if the vector of Lagrange multipliers is
sparse, meaning that only a few box constraints are active and most
Lagrange multipliers are zero.  As soon as the active set is
identified, Krylov convergence occurs. We then converge rapidly to the
solution.

Additionally, Krylov convergence can be further accelerated with the
help of a preconditioner. In this case, we use residuals that are the solution
of {$Mr_k = A^T(Ax_k-b)-\lambda_k + \mu_k$} for some nonsingular
matrix $M$ that is cheap to solve.  The residuals are then
$M$-orthogonal.

However, there are many inverse problems where numerous variables
reach the bounds. Consider, for example, the deblurring of a satellite
picture against a black background.  In such cases, the Lagrange
multipliers are not sparse, but the deviation from the bounds on $x$
is sparse.  Therefore, a dual ResQPASS algorithm would be
helpful. Instead of expanding $x$ in the basis of residuals, we can
expand the Lagrange multipliers, which will lead to a sparse $x$.
Preliminary experiments confirm this intuition, and the dual algorithm
demonstrates similar superlinear convergence. The analysis of the dual
algorithm, where we use the residuals as basis to expand the Lagrange
multipliers, is the subject of a future paper.

Additional applications of this methodology can be found in contact
problems in mechanics, where the subproblems interact at only a few
points. These points of contact are often unknown upfront.

Future work is a thorough analysis of the propagation of
rounding errors and effects of the loss of orthogonality.

At present, we have not leveraged the asymptotic Krylov structure. If
a robust detection mechanism exists to determine whether we are in the
asymptotic regime, the Gram-Schmidt factorization of the Hessian could
be simplified to a bidiagonalization.

\section*{Acknowledgments}
We thank Jeffrey Cornelis for fruitful discussions during the initial
phase of the research.  Wim Vanroose acknowledges fruitful discussions
with Przemysław Kłosiewicz concerning an earlier version of QPAS
during earlier project (unpublished).  Bas Symoens acknowledges
financial support from the Department of Mathematics, U. Antwerpen.
We acknowledge the anonymous referees for constructive feedback and
the recommendation to apply QR as in \Cref{eq:qr}, instead of the product
of the right-hand side. \bibliographystyle{siamplain} \bibliography{biblio}

\begin{thebibliography}{10}

\bibitem{chen2001atomic}
{\sc S.~S. Chen, D.~L. Donoho, and M.~A. Saunders}, {\em Atomic decomposition
  by basis pursuit}, SIAM review, 43 (2001), pp.~129--159.

\bibitem{cools2018analyzing}
{\sc S.~Cools, E.~F. Yetkin, E.~Agullo, L.~Giraud, and W.~Vanroose}, {\em
  Analyzing the effect of local rounding error propagation on the maximal
  attainable accuracy of the pipelined conjugate gradient method}, SIAM Journal
  on Matrix Analysis and Applications, 39 (2018), pp.~426--450.

\bibitem{dostal2016scalable}
{\sc Z.~Dost{\'a}l, T.~Kozubek, M.~Sadowsk{\'a}, and V.~Vondr{\'a}k}, {\em
  Scalable algorithms for contact problems. {With} contributions by
  {Tom{\'a}{\v{s}}} {Brzobohat{\'y}}, {David} {Hor{\'a}k}, {Lubom{\'{\i}}r}
  {\v{R}}{\'{\i}}ha, {Old{\v{r}}ich} {Vlach}}, vol.~36 of Adv. Mech. Math.,
  Cham: Birkh{\"a}user, 2nd updated and expanded edition~ed., 2023,
  \url{https://doi.org/10.1007/978-3-031-33580-8}.

\bibitem{gazzola2021flexible}
{\sc S.~Gazzola}, {\em Flexible {CGLS} for box-constrained linear least squares
  problems}, in 2021 21st International Conference on Computational Science and
  Its Applications (ICCSA), IEEE, 2021, pp.~133--138.

\bibitem{gazzola2017fast}
{\sc S.~Gazzola and Y.~Wiaux}, {\em Fast nonnegative least squares through
  flexible {K}rylov subspaces}, SIAM Journal on Scientific Computing, 39
  (2017), pp.~A655--A679.

\bibitem{ghysels2014hiding}
{\sc P.~Ghysels and W.~Vanroose}, {\em Hiding global synchronization latency in
  the preconditioned conjugate gradient algorithm}, Parallel Computing, 40
  (2014), pp.~224--238.

\bibitem{gillis2020nonnegative}
{\sc N.~Gillis}, {\em Nonnegative matrix factorization}, SIAM, 2020.

\bibitem{golub1965calculating}
{\sc G.~Golub and W.~Kahan}, {\em Calculating the singular values and
  pseudo-inverse of a matrix}, Journal of the Society for Industrial and
  Applied Mathematics, Series B: Numerical Analysis, 2 (1965), pp.~205--224.

\bibitem{gondzio2012interior}
{\sc J.~Gondzio}, {\em Interior point methods 25 years later}, European Journal
  of Operational Research, 218 (2012), pp.~587--601.

\bibitem{gu2023fast}
{\sc R.~Gu, S.~J. Billinge, and Q.~Du}, {\em A fast two-stage algorithm for
  non-negative matrix factorization in smoothly varying data}, Acta
  Crystallographica Section A: Foundations and Advances, 79 (2023).

\bibitem{hammarling2008updating}
{\sc S.~Hammarling and C.~Lucas}, {\em Updating the {QR} factorization and the
  least squares problem},  (2008).

\bibitem{hansen2006deblurring}
{\sc P.~C. Hansen, J.~G. Nagy, and D.~P. O'leary}, {\em Deblurring Images:
  Matrices, Spectra, and Filtering}, SIAM, 2006.

\bibitem{hastie2015matrix}
{\sc T.~Hastie, R.~Mazumder, J.~D. Lee, and R.~Zadeh}, {\em Matrix completion
  and low-rank {SVD} via fast alternating least squares}, The Journal of
  Machine Learning Research, 16 (2015), pp.~3367--3402.

\bibitem{kak2001principles}
{\sc A.~C. Kak and M.~Slaney}, {\em Principles of Computerized Tomographic
  Imaging}, SIAM, 2001.

\bibitem{kim2011fast}
{\sc J.~Kim and H.~Park}, {\em Fast nonnegative matrix factorization: An
  active-set-like method and comparisons}, SIAM Journal on Scientific
  Computing, 33 (2011), pp.~3261--3281.

\bibitem{lawson1995solving}
{\sc C.~L. Lawson and R.~J. Hanson}, {\em Solving Least Squares Problems},
  SIAM, 1995.

\bibitem{liesenstrakos}
{\sc J.~Liesen and Z.~Strakos}, {\em Krylov Subspace Methods: Principles and
  Analysis}, Oxford University Press, 2013.

\bibitem{lubbecke2005selected}
{\sc M.~E. L{\"u}bbecke and J.~Desrosiers}, {\em Selected topics in column
  generation}, Operations Research, 53 (2005), pp.~1007--1023.

\bibitem{nagy2000enforcing}
{\sc J.~G. Nagy and Z.~Strakos}, {\em Enforcing nonnegativity in image
  reconstruction algorithms}, in Mathematical Modeling, Estimation, and
  Imaging, vol.~4121, SPIE, 2000, pp.~182--190.

\bibitem{nocedal1999numerical}
{\sc J.~Nocedal and S.~J. Wright}, {\em Numerical Optimization}, Springer,
  1999.

\bibitem{paige1982lsqr}
{\sc C.~C. Paige and M.~A. Saunders}, {\em {LSQR:} an algorithm for sparse
  linear equations and sparse least squares}, ACM Transactions on Mathematical
  Software (TOMS), 8 (1982), pp.~43--71.

\bibitem{stark1995bounded}
{\sc P.~B. Stark and R.~L. Parker}, {\em Bounded-variable least-squares: an
  algorithm and applications}, Computational Statistics, 10 (1995),
  pp.~129--129.

\bibitem{van2003iterative}
{\sc H.~A. Van~der Vorst}, {\em Iterative {K}rylov Methods for Large Linear
  Systems}, no.~13, Cambridge University Press, 2003.

\bibitem{vanroose2021krylov}
{\sc W.~Vanroose and J.~Cornelis}, {\em Krylov-simplex method that minimizes
  the residual in $\ell_1$-norm or $\ell_\infty$-norm}, arXiv preprint
  arXiv:2101.11416,  (2021).

\end{thebibliography}
\end{document}